\documentclass[11pt]{amsart}
\usepackage[latin1]{inputenc}
\usepackage[english]{babel}
 \usepackage{cite}
 \usepackage{relsize}
 \usepackage{bm}
 \usepackage{bibentry}
 \usepackage[all]{xypic}
\usepackage{amsmath, amsfonts, amssymb, amsthm, mathtools, mathrsfs,bm}
\usepackage{enumerate}
\usepackage{setspace}

\def\Hom{\mathop{\rm Hom}\nolimits}
\def\End{\mathop{\rm End}\nolimits}

\usepackage[pagebackref,colorlinks,linkcolor=blue,citecolor=blue,urlcolor=blue,hypertexnames=true]{hyperref}
\usepackage[pagebackref,hypertexnames=true]{hyperref}
\usepackage[usenames,dvipsnames]{xcolor}
\definecolor{darkpastelred}{rgb}{0.76, 0.23, 0.13}
\definecolor{darkred}{rgb}{0.55, 0.0, 0.0}
\definecolor{darkmagenta}{rgb}{0.55, 0.0, 0.55}
\definecolor{coolblack}{rgb}{0.0, 0.18, 0.39}
\definecolor{ceruleanblue}{rgb}{0.16, 0.32, 0.75}

\newtheorem{theorem}{Theorem}[section]
\newtheorem{lemma}[theorem]{Lemma}
\newtheorem{proposition}[theorem]{Proposition}
\newtheorem{corollary}[theorem]{Corollary}

\theoremstyle{definition}
\newtheorem{definition}[theorem]{Definition}

\newtheorem{remark}[theorem]{Remark}
\newtheorem{notation}[theorem]{Notation}

\newcommand{\Bun}{\mathrm{Bdl}}

\newcommand{\G}{\Gamma}

\newcommand{\CR}{\mathcal{R}}
\newcommand{\QQ}{\mathcal{Q}}

\newcommand{\ep}{\varepsilon}

\newcommand{\N}{\mathbb{N}}
\newcommand{\Z}{\mathbb{Z}}
\newcommand{\Q}{\mathbb{Q}}
\newcommand{\TT}{\mathbb{T}}
\newcommand{\T}{\mathbb{T}}
\newcommand{\R}{\mathbb{R}}
\newcommand{\C}{\mathbb{C}}
\newcommand{\CCC}{\mathbb{C}}
\newcommand{\id}[1]{{\rm id}_{#1}}

\begin{document}

\title{Non-stable groups }
\author{Marius Dadarlat}\address{MD: Department of Mathematics, Purdue University, West Lafayette, IN 47907, USA}\email{mdd@purdue.edu}	
\begin{abstract} In this article we discuss cohomological obstructions to two kinds of group stability.
In the first part, we show that  residually finite groups $\Gamma$ which arise as  fundamental groups of compact Riemannian manifolds with strictly negative sectional curvature  are not uniform-to-local stable with respect to the operator norm if their even Betti numbers  $b_{2i}(\Gamma)$ do not vanish. In the second part, we show that  non-vanishing of Betti numbers $b_{i}(\Gamma)$ in dimension $i>1$  obstructs  $C^*$-algebra stability for  groups approximable by unitary matrices
  that admit a coarse embedding in a Hilbert space.

\end{abstract}

\thanks{This research was partially supported by NSF} 
\date{\today}
\maketitle
%\tableofcontents
\section{Introduction}

For a countable discrete group $\G$ we consider several natural stability properties relative to unitary groups $U(n)$ equipped with the uniform norm. We also consider stability properties of  $\G$ with respect to  unitary groups  of $C^*$-algebras. The reader is referred to the survey papers by  Arzhantseva \cite{Arzhantseva} and  Thom \cite{Thom:ICM} for an introduction to the approximation and stability properties
of groups. Just like in our earlier paper on this subject  \cite{CCC}, we rely on ideas  due to Kasparov~\cite{Kas:inv},  Connes, Gromov and Moscovici \cite{CGM:flat},  Gromov \cite{Gromov:reflections, Gromov:curvature},
 Tu\cite{Tu:gamma} and Kubota~\cite{Kubota2}. In addition we use results of Hanke and Schick \cite{Hanke}, \cite{Hanke-Schick}, Hunger~\cite{Hunger} and  Baird and Ramras \cite{Ramras}.

\begin{definition}\label{def:main} For a countable discrete group $\G$ we consider
sequences  $\{\rho_n\}$ of unital maps  and sequences  $\{\pi_n\}$ of unitary group representations  with $\rho_n,\pi_n  :\Gamma \to U(n).$

The sequence $\{\pi_n\}$ approximates $\{\rho_n\}$ \emph{locally} if

 \[\lim_{n\to \infty}\, \|\rho_n(s)-\pi_n(s)\|=0,  \,\, \text{for all}\,\,\, s\in \Gamma.\]

 The sequence $\{\pi_n\}$ approximates $\{\rho_n\}$ \emph{uniformly} if

\[{ \lim_{n\to \infty} \,\sup_{s\in \Gamma} \|\rho_n(s)-\pi_n(s)\|=0}.\]

\begin{itemize}
	\item [(a)] $\G$ is \emph{locally stable} if  any sequence $\{\rho_n\}$ which satisfies
\begin{equation*}\label{U}
{ \lim_{n\to \infty} \,\,\|\rho_n(st)-\rho_n(s)\rho_n(t)\|=0}, \,\, \text{for all}\,\, s,t \in \G
\end{equation*}
can be approximated   \emph{locally}  by a sequence $\{\pi_n\}$	 of unitary representations.
\item [(b)]  $\Gamma$ is \emph{uniformly stable}
if  any sequence $\{\rho_n\}$ which satisfies
\begin{equation*}\label{U}
{ \lim_{n\to \infty} \,\,\sup_{s,t\in \Gamma} \|\rho_n(st)-\rho_n(s)\rho_n(t)\|=0}
\end{equation*}
can be approximated  \emph{uniformly}   by a sequence $\{\pi_n\}$	 of unitary  representations.
\item [(c)]  $\Gamma$ is \emph{uniform-to-local stable}
if  any sequence $\{\rho_n\}$  which satisfies  the assumption from (b)
can be approximated  \emph{locally}   by a sequence $\{\pi_n\}$	 of  unitary representations.

\item [(d)]  $\Gamma$ is \emph{local-to-uniform stable}
if  any sequence $\{\rho_n\}$ which satisfies  the assumption from (a)
can be approximated  \emph{uniformly}   by a sequence $\{\pi_n\}$	 of unitary  representations.
\end{itemize}
%\item[]
 One may visualize the conditions (a), (b) and (c) in a diagram:
\[
\xymatrix{
\left(\sup\limits_{s,t\in \Gamma} \|\rho_n(st)-\rho_n(s)\rho_n(t)\|\to 0\right)\ar@{=>}[rr]^-{\text{uniformly stable}}  \ar@{=>}[drr]^-{\hskip 2cm \text{uniform-to-local  stable}}& &
{\left(\sup\limits_{s\in \Gamma}\|\rho_n(s)-\pi_n(s)\|\to 0\right)}
\\
{ \|\rho_n(st)-\rho_n(s)\rho_n(t)\|\to 0}_{\,\, \forall s,t}
 \ar@{=>}[rr]_-{\text{locally stable}} & & {{\|\rho_n(s)-\pi_n(s)\|\to 0}_{\,\, \forall s} }
}
\]
\end{definition}
\vskip 8pt

The stability properties  considered in Definition~\ref{def:main} are quite different in nature.
Local-to-uniform stability is not too interesting for it is satisfied only by finite groups.    While it is clear that  both uniform stability and local stability imply uniform-to-local stability, the study of  these three properties is more challenging.
Local stability is referred to as \emph{matricial stability} in the paper of Eilers, Shulman and S{\o}rensen \cite{ESS-published}. We showed that the nonvanishing of rational cohomology in even dimensions is an obstruction to local stability for large classes of groups, including all amenable groups and all linear groups, \cite{CCC} and \cite{DDD}.
  Bader,  Lubotzky, Sauer  and Weinberger  showed that lattices in semisimple real Lie groups are typically not locally stable \cite{BLSW}.

Uniform stability is called  \emph{Ulam stability}   in the article of  Burger, Ozawa and Thom \cite{BOT}. By a classical result of Kazhdan  discrete amenable groups are uniformly stable, \cite{Kazhdan-epsilon}.
In contrast, an amenable group $\G$ is not locally stable if $H^{2i}(\G,\R)\neq 0$ for some $i>1$, \cite{CCC}.  A recent  paper of Glebsky, Lubotzky, Monod and Rangarajan \cite{Ulam-new} studies uniform stability for cocompact lattices in semisimple Lie groups by means of asymptotic cohomology.

 Uniform stability is a stronger condition than  uniform-to-local stability.  Indeed,
 if $\G$ has a finite index subgroup isomorphic to
a free group $\mathbb{F}_k$, $k\geq 2,$ then $\G$ is locally stable \cite{ESS-published}, \cite{BLSW} and hence uniform-to-local  stable, but
 $\G$ is not uniformly stable.
More generally, it was shown in  \cite{BOT} that if the comparison map
%\begin{equation}\label{com}
$j:H^2_b(\Gamma, \R) \rightarrow H^2(\Gamma, \R)$
%\end{equation}
is not injective, then $\Gamma$ is not uniformly stable, and consequently,
 the non-elementary hyperbolic groups are not uniformly stable, since $j$ is not injective for such groups \cite{EpFu}.

The surface groups $\Gamma_g$ of genus $g>1$ were the first groups shown not be uniformly stable, \cite{Kazhdan-epsilon}.
One observes that ~Kazhdan's  proof, which  exploits the nonvanishing of $H^2(\Gamma_g,\R)$,  shows that, in fact,   the groups $\Gamma_g$ of genus $g >1$
are not even uniform-to-local stable. Motivated by this observation,  in the first part of this paper, we point out that many of the cocompact lattices in the Lorentz group $SO_0(n,1)$, $n>1$ are not uniform-to-local stable.
This will follow from the following theorem inspired by an idea of Gromov~\cite[p.166]{Gromov-large}:

\begin{theorem}\label{thm:main-result}
  Let $M$ be a closed connected Riemannian manifold with strictly negative sectional curvature  and residually finite fundamental group. If  $b_{2i}(M)\neq  0$ for some $i>0$, then $\pi_1(M)$ is not uniform-to-local stable.
\end{theorem}
Theorem~\ref{thm:main-result} is a direct consequence of Theorem~\ref{thm:main-resultA}  from Section~\ref{sect:4}.
Concerning the assumption on Betti numbers, observe that  if $M$ is orientable and  $\dim M=2m$ then $b_{2m}(M)=1$ while  if $M$ is orientable and $\dim M=2m+1$,  then it suffices to require that
 $b_i(M)> 0$ for some $1 \leq i \leq 2m,$ since $b_i(M)=b_{2m+1-i}(M)$ by Poincar{\'e} duality and  either $i$ or $2m+1-i$ must be even.

Recall that a compact lattice in a semisimple real Lie group $G$  is a discrete subgroup $\G$ of $G$ such that the quotient space $G/\G$ is compact.
The $n$-dimensional hyperbolic space $\mathbb{H}^n$, $n\geq 2,$ has constant sectional curvature equal to  $-1$. The connected component of the identity of the group of orientation preserving isometries of  $\mathbb{H}^n$ is the Lorentz group $SO_0(n,1)$. Since $\mathbb{H}^n$ is isometric to the symmetric space $SO_0(n,1)/ SO(n)$, if $\G$ is a torsion free cocompact lattice in $SO_0(n,1),$
then $M=\Gamma\setminus SO_0(n,1)/ SO(n)$ is an orientable closed connected Riemannian manifold with sectional curvature $=-1$.  Moreover, $\G=\pi_1(M)$  is finitely generated  by co-compactness and hence it is residually finite by Malcev's theorem since $SO_0(n,1)\subset \mathrm{GL}(n,\R)$. Thus one obtains the following:
\begin{corollary}\label{cor:Lorentz}
  Let $\Gamma$ be a torsion free cocompact lattice in $SO_0(n,1)$.
  \begin{itemize}
  \item[(i)] If $n$ is even then $\Gamma$ is not uniform-to-local stable.
  \item [(ii)] If $n$ is odd and  $b_{i}(\Gamma)> 0$ for some $i>0$ then $\Gamma$ is not uniform-to-local stable.
  \end{itemize}
 \end{corollary}
The  corollary reproves Kazhdan's result since $\Gamma_g \subset SO_0(2,1)$, $g>1$. In order to apply the corollary to other examples, let us note that by Selberg's lemma, any cocompact discrete subgroup $\Lambda$ of $SO_0(n,1)$
  has a finite index torsion free subgroup $\G$.
 Thus, in order to apply Corollary~\ref{cor:Lorentz}(ii)
 to $\G,$ it remains to realize the condition on Betti numbers.

It was shown in  \cite{BLSW} that if $\G$ is a cocompact lattice in a real semisimple Lie group $G$ which is not locally isomorphic to either $SO(n,1)$ for $n$ odd or $SL_3(\R)$, then $b_{2i}(\G)>0$ for some $i>0$, so that $\G$ is not locally stable by \cite{CCC}.
 Concerning lattices in $SO(n,1)$,  the following  nonvanishing result is established in \cite{BLSW}.
   %\begin{theorem}\label{thm:BLSW}
 %Let $\Lambda$ be a cocompact lattice in $SO(n,1)$ with $n>1$ odd.
%Suppose that one of the following conditions is satisfied
% \begin{itemize}
%  \item[(i)] $n=3$
%  \item [(ii)] $n=4m+1$
%    \item [(iii)] $n=4m+3$ and $\Lambda$ is arithmetic but not of the form ${}^6D_4$ if $n=7$
%  \end{itemize} Then there is a finite index subgroup $\Lambda_0\leq\Lambda$ such that for
%  finite index subgroup $\G\leq\Lambda_0$ there an even $i>0$ such that $b_{2i}(\G)>0$.
% \end{theorem}
 \begin{theorem}[Cor.\,3.13 of \cite{BLSW}]\label{thm:BLSW}
 Let $\Lambda$ be a cocompact lattice in $SO(n,1)$ with $n>1$ odd.
Suppose either (i) $n=3$ or (ii) $n=4m+1$ or (iii) $n=4m+3$ and $\Lambda$ is arithmetic but not of the form ${}^6D_4$ if $n=7.$
  Then there is a finite index subgroup $\Lambda_1\leq\Lambda$ such that for any
  finite index subgroup $\G\leq\Lambda_1$ there is $i>0$ such that $b_{2i}(\G)>0$. In particular, the group $\G$ is not locally stable.
 \end{theorem}
By Theorem~\ref{thm:main-result}, we deduce:
  \begin{corollary}\label{cor:Lorentzz}
  The groups $\G$ as in Theorem~\ref{thm:BLSW} are not  uniform-to-local stable.
 \end{corollary}
  For concrete examples one may consider
 $G=SO_0(x_1^2+\cdots +x_{n}^2-\sqrt{p}\, x_{n+1}^2,\R)\cong SO_0(n,1),$ where $p$ is a square free integer.  Let  $\mathcal{O}$ be ring of integers of $\mathbb{Q}\sqrt{p}$.   Then $\mathcal{O}=\Z+\Z\sqrt{p}$ if $p\not\equiv 1$ (mod 4) and $\mathcal{O}=\{\frac{a+b\sqrt{p}}{2}\colon a,b \in \Z,\, a-b \equiv 0\, (\text{mod 2})\}$ if $p\equiv 1$ (mod 4) and
$G_{\mathcal{O}}$ is a cocompact arithmetic lattice in $G$, \cite{Borel}.
By a result of Li and Millson~\cite{Li-Millson}, any arithmetic lattice  in $SO_0(n,1)$, $n\neq 3,7$ contains a congruence subgroup $\G$ such that $b_1(\G)>0$.

\vskip 8pt

In the second part of the paper we revisit local stability and discuss $C^*$-stability  of discrete groups \cite{ESS-published}, a property which can be viewed as local stability relative to $C^*$-algebras.

\begin{definition}\label{def: C*stable}
  A  group $\Gamma$ is \emph{$C^*$-stable}
 if for any sequence  of  unital maps
  $\{\rho_n:\Gamma \to U(B_n)\}_{n}$, where $B_n$ are unital $C^*$-algebras   such that
 \[
{ \lim_{n\to \infty}  \|\rho_n(st)-\rho_n(s)\rho_n(t)\|=0}, \quad \text{for all}\,\, s,t\in \G,\]
there exists a sequence of  group homomorphisms  $\{\pi_n:\Gamma \to U(B_n)\}_{n}$ satisfying
\[{\lim_{n\to \infty}\, \|\rho_n(s)-\pi_n(s)\|=0}, \quad \text{for all}\,\, s\in \G.\]
\end{definition}
We note that local stability corresponds to $C^*$-stability relative to finite dimensional $C^*$-algebras.
As discussed earlier, nonvanishing of even-dimensional  rational cohomology is an obstruction to matricial stability for many groups.
 Prompted by a question of Dima Shlyakhtenko concerning the possible role of  odd-dimensional cohomology in group stability, we show the following:
  \vskip 4pt
  \begin{theorem}\label{thm-cstable}
{Let $\Gamma$ be a countable discrete MF-group  that admits a $\gamma$-element.
  If $H^{k}(\Gamma,\mathbb{Q})\neq 0$ for some $k>1$,
  then $\Gamma$  is not $C^*$-stable.}
  \end{theorem}
Let us recall that
 a group $\G$ is MF if it is isomorphic to a subgroup of the unitary group of the corona $C^*$-algebra
  $\prod_n M_{n}/\bigoplus_n M_{n},$ \cite{CDE2013}.  Equivalently, $\G$ embeds in $\mathbf{U}/\mathbf{N}$ where $\mathbf{U}=\prod_{n=1}^{\infty} U(n)$ and
  $\mathbf{N}=\{(u_n)_n \in \mathbf{U}\,:\, \|u_n-1_n\|\to 0\}$.
In other words a group is MF if it admits sufficiently many approximate unitary representations to  effectively separate its elements. In the terminology of \cite{DECHIFFRE} these are the $\big(U(n), \|\cdot\|\big)_{n=1}^{\infty}$-approximated groups.
It is an open problem to find examples of discrete countable groups which are not MF. The groups that are locally embeddable in amenable groups  are MF  as a consequence of \cite{TWW}.

The class of groups that admit a $\gamma$-element is  large \cite{Kasparov-Skandalis-kk}, \cite{Kasparov-Skandalis-Annals}. It includes the groups that admits a uniform embedding in a Hilbert space \cite{Tu:gamma}.
The amenable groups, or more generally, the groups with Haagerup's property are uniformly embeddable in a Hilbert  space \cite{Valette-book}  and so are the linear groups \cite{GuenHW}.
Hilbert space uniform embeddability
passes to  subgroups  and  products,  direct  limits,  free  products with amalgam, and extensions by exact groups \cite{Dad-Guen}.
\vskip 4pt

 Theorem~\ref{thm-cstable} will be established as a consequence of a result which assumes  weaker forms of stability, see Theorem~\ref{thm:2nd} and Theorem~\ref{cor2nd}. More precisely, Theorem~\ref{thm-cstable} follows from Theorem~\ref{cor2nd}(ii), where only stability with respect to $C^*$-algebras of the form $M_n(C(\TT))$ is assumed. The proof proceeds by an adaptation of the arguments from \cite{CCC} (and we repeat many of them here for the sake of completion) with the novelty that one employs a theorem of Baird and Ramras \cite{Ramras} on vanishing of Chern classes for families of flat bundles in place of a result of Milnor \cite{Milnor}.

 Moreover, we show in Theorem~\ref{thm:2nd} that  if a quasidiagonal groups $\G$ admits a $\gamma$-element  and has a nonvanishing  Betti number $b_{2k}(\G),$ then there are sequences $\{\rho_n\}$ as in Definition~\ref{def:main}(a) which cannot be locally approximated  even if we allow for non-unitary representations $\pi_n : \G \to GL_n(\C)$. Similarly, for  a finitely generated  group $\G$ as  in Theorem~\ref{thm-cstable}, there is a sequence of unital maps $\{\rho_n:\G \to U_n(C(\TT))\}$ as in Definition~\ref{def: C*stable} which cannot be approximated locally by a sequence of homomorphisms $\{\pi_n:\G \to GL_n(C(\TT))\}$.

For the sake of accessibility, we include in our exposition several facts well-known to the experts. In Sections~\ref{sect:2} and \ref{sect:3} we revisit the topics of  flat bundles and almost flat bundles. The proof of Theorem~\ref{thm:main-result} which relies on the approximate monodromy correspondence for almost flat bundles is given in  Section~\ref{sect:4}.  The proof of Theorem~\ref{thm-cstable} is given in Section~\ref{sec:semi}.
\section{Flat bundles}\label{sect:2}
Let $A$ be a unital $C^*$-algebra and let $V$ be a finitely generated (projective) right Hilbert $A$-module.  Let $\mathcal{L}(V)$ be the  $C^*$-algebra of adjointable $A$-linear operators acting on $V$.
The unitary group of $\mathcal{L}(V)$ will be denoted by $U(V)$.
For a compact Hausdorff space $M$ we denote by $\Bun_A^V(M)$ the set of isomorphism classes of locally trivial bundles with fiber $V$ and structure group $U(V)$. If $V=A$ we write $\Bun_A(M)$ for $\Bun_A^A(M).$
 If $E\in \Bun_A^V(M)$ we say that $E$ is a Hilbert $A$-module bundle with typical fiber $V$.
If $M$ is a smooth manifold, then every $E\in \Bun_A^V(M)$  admits a smooth structure which unique up to isomorphism \cite[Thm.3.14]{Schick:ny}. The $C^\infty(M,A)$-module of smooth sections of $E$ is denoted by $C^\infty(E)$.
If $A=\C$ and $V=\C^r$, we write $\Bun_\C^r(M)$ for $\Bun_\C^{\C^r}(M)$,  the set of isomorphism classes of hermitian complex vector bundles of rank $r$.

\begin{definition}
 A flat structure on a smooth bundle $E\in \Bun_A^V(M)$ is given by a finite cover $\mathcal{U}=(U_i)_{i\in I}$ of $M$ together with smooth trivializations
$U_i\times V \to E_{U_i}$  such that  the corresponding transitions functions $v_{ij}:U_i\cap U_j \to U(V)$ are constant functions.
\end{definition}
The classic theory of connections  on vector bundles extends to smooth Hilbert $A$-module bundles as discussed in \cite[Sec.3]{Schick:ny}.
Let $E$ be smooth Hilbert $A$-module bundle over   a  Riemannian manifold $M$. A  connection on $E$ is a $\C$-linear map
 $$\nabla:C^\infty(TM\otimes\C)\otimes C^\infty(E)\to  C^\infty(E),\quad X \otimes s \mapsto \nabla_X(s)$$    which satisfies the conditions
  \begin{itemize}
  \item[(i)] $\nabla_X(s\cdot \mathbf{f})=s\cdot(X\mathbf{f}) +\nabla_X(s)\cdot \mathbf{f}, \quad (\text{Leibnitz formula})$
  \item[(ii)] $\nabla_{fX}(s)=f\nabla_X(s)$
  \end{itemize}
  for every $X\in C^\infty(TM\otimes\C)$, $f\in C^\infty(M,\C)$, $\mathbf{f}\in C^\infty(M,A)$ and $s\in C^\infty(E)$.
    The connection $\nabla$ is compatible with the metric of $E$ if
$$ X\langle s,s'\rangle=\langle \nabla_Xs,s'\rangle+\langle s,\nabla_Xs'\rangle.$$
The importance of having a connection is that allows one to lift smooth paths $\gamma(t)$ between points $p,q\in M$ to  isomorphisms between  fibers $E_p$ and $E_q$ via parallel transport $P_\gamma:E_p \to E_q$.
Recall that a section $s$ along $\gamma$  is parallel if it satisfies the differential equation $\nabla_{{\partial_t\gamma}}s=0$.
This equation has a unique solution for each initial value
 $s(p)\in E_p$ and thus determines uniquely the value of $s(q)\in E_q$ so that one defines $P_\gamma(s(p))=s(q)$.
Compatibility of $\nabla$ with the metric implies that $P_\gamma$ is a unitary operator, this is why such a connection
 $\nabla$ is also called a unitary connection.

The curvature of $\nabla$ is the tensor $R^{\nabla}\in \Omega^2(M, \End{E})$ defined by the equation $$R^{\nabla}(X,Y)=\nabla_X\nabla_Y-\nabla_Y\nabla_X -\nabla_{[X,Y]}.$$

\begin{definition} A couple $(E,\nabla)$ consisting of a bundle $E\in \Bun_A^V(M)$ and a unitary connection is flat if
 the curvature of the unitary connection vanishes, $R^{\nabla}=0$.
\end{definition}
For our discussion of flat bundles, it is convenient to adopt a setup from \cite[Sec.3]{Hanke}.
Let $\mathcal{P}_1(M)$ be the path groupoid of $M$ with objects points of $M$ and morphisms $\mathcal{P}_1(M)(p,q)$ the piecewise smooth paths $[0,1]\to M$ connecting $p$ to $q$.
The product $\gamma\cdot \gamma'$ of two paths is  the path obtained by first traversing $\gamma'$ and then $\gamma$, thus it is given by the concatenation $\gamma'*\gamma$.
 One endows
$\mathcal{P}_1(M)$ with its natural topology. As it is usual, we let $\Omega_1(M,p)$ stand for $\mathcal{P}_1(M)(p,p).$
Denote by ${\Pi}_1(M)$  the fundamental groupoid of $M$ obtained from $\mathcal{P}_1(M)$  by taking homotopy classes of paths with fixed endpoints and let $\G=\pi_1(M,p)$ be the fundamental group of $M$.

For a bundle $E\in\Bun^V_A(M)$, we denote  by $\mathcal{T}(E)$  the transport groupoid of $E$ with objects the points of $M$ and  morphisms
$\mathcal{T}(E)(p,q)=\mathrm{Isom}_A(E_p,E_q)$. Following \cite[p.288]{Hanke}, we endow the groupoid $\mathcal{T}(E)$
with its natural topology, where the set of morphisms is topologized by
using local trivializations in order to identify nearby fibers of $E$ and  $\mathrm{Isom}_A(E_p,E_q)$ is given the uniform norm. A holonomy representation on the bundle $E$ is a continuous morphism of groupoids
\begin{equation}\label{holy}
h:\mathcal{P}_1(M) \to \mathcal{T}(E)
\end{equation}
By a classic result, if $E$ is smooth and it is endowed with a unitary connection $\nabla$,
then the corresponding parallel transport satisfies $P_{\gamma\cdot \gamma'}=P_\gamma \circ P_{\gamma'}$ and hence it defines a holonomy representation $\gamma\mapsto h(\gamma)=P_\gamma$, in the sense discussed above, \cite[Thm.9.8]{smooth}.

\begin{definition}\label{def} Let $\pi:\Gamma=\pi_1(M,p) \to U(V)$ be a group homomorphism. The universal cover of $M$ is denoted by $\widetilde{M}$. One realizes $\widetilde{M}$ as a space of homotopy classes of curves $\eta:[0,1]\to M$ with $\eta(0)=p$ and homotopies preserving the endpoints. The left action of  $G$ on $\widetilde{M}$ is defined as follows. If $s\in \pi_1(M,p)$ is represented by a loop
 $\gamma\in\Omega_1(M,p)$, then $s\cdot [\eta] =[\eta\cdot \gamma^{-1}]$ is represented by  the path given by traversing $\gamma$ in opposite direction followed by traversing $\eta$.
The orbit space of the left action of  $\G$ on $\widetilde{M}\times V$, defined by $s\cdot ([\eta],v)=(s\cdot [\eta],\pi(s)v)$,
  is the total space of  a (flat) Hilbert A-module bundle $\widetilde{M}\times_\pi V \to M$ denoted  by $L_\pi\in \Bun_A^V(M)$.
 The map $$\pi \mapsto L_{\pi}$$  was introduced by  Atiyah \cite{atiyah-map} in the context of finite group representations.

If $A=C^*(\G)$ and $\jmath:\G\hookrightarrow U(C^*(\G))$ is the natural inclusion, then $L_{\jmath}$ is a bundle of free rank-one Hilbert $C^*(\G)$-modules. The bundle $L_{\jmath}\in \Bun_{C^*(\G)}(M)$ is called Mishchenko's flat bundle.
 \end{definition}
 \begin{proposition}[\cite{Kobayashi}]\label{prop:koba} For a smooth  bundle $E\in \Bun^V_A(M)$  the following are equivalent:
\begin{itemize}
 \item[(i)] $E$ admits a  flat structure.
 \item[(ii)] $E$ admits a unitary connection $\nabla$ with zero curvature, $R^\nabla=0$, i.e. $(E,\nabla)$ is flat.
 \item[(iii)] $E$ is defined by a representation $\pi:\pi_1(M)\to U(V)$ in the sense that  $E\cong L_{\pi}$.
 \item[(iv)] There is a holonomy representation $h:\mathcal{P}_1(M) \to \mathcal{T}(E)$ which descends to a morphism of groupoids
 $\mathbf{h}:{\Pi}_1(M) \to \mathcal{T}(E)$
\end{itemize}
\end{proposition}
\begin{proof} This is proved in \cite{Kobayashi} (see Prop.1.2.5 and 1.4.21) for complex vector bundles. The same arguments work without change in the present context. The holonomy representation ${h}$ is given by  parallel transport with respect to $\nabla$. Restriction of  $\mathbf{h}$ to $\pi_1(M,p)$ gives a representation $\pi$ as in (iii).
\end{proof}

It is also convenient to  work with selfadjoint idempotents $e$ in  matrices $m\times m$ over the $C^*$-algebra $C(M)\otimes A$
that represent bundles $E\in \Bun_A^V(M)$, where $V\cong e(p)A^m$, $p\in M$.
\begin{notation}\label{not:misi}
 Fix a flat structure for the Mishchenko's bundle  $L_{\jmath}$  given by some finite cover $\mathcal{U}=(U_i)_{i\in I}$ of $M$ together with smooth trivializations
$U_i\times C^*(\G) \to (L_{\jmath})_{U_i}$  such that all the corresponding transitions functions $s_{ij}:U_i\cap U_j \to \G\subset U(C^*(\G))$ are  constant. Thus one obtains group elements $s_{ij}\in \G$
 which form  a 1-cocycle:  $s_{ij}^{-1}=s_{ji}$ and $s_{ij} \cdot s_{jk}=s_{ik}$ whenever $U_i\cap U_j \cap U_k \neq \emptyset.$
 Let $(\chi_i)_{i\in I}$
be positive smooth functions with $\chi_i$ supported in $U_i$ and such that
$\sum_{i\in I} \chi^2_i=1$. Set $m=|I|$ and let $(e_{ij})$ be the canonical matrix unit of $M_m(\CCC)$. Then $L_{\jmath}\in \Bun_{C^*(\G)}(M)$ is represented by the selfadjoint projection
\begin{equation}\label{eqn:mish1}
 \ell_{\jmath}=\sum_{i,j\in I} e_{ij}\otimes \chi_i\chi_j\otimes s_{ij}\in M_{m}(\CCC)\otimes C(M)\otimes C^*(\G).
 \end{equation}
 Moreover, the bundle $L_{\pi}\in \Bun_A^V(M)$, corresponding to the group homomorphism $\pi:\Gamma \to U(V)$ (see Definition~\ref{def}),  is represented by the selfadjoint projection
\begin{equation}\label{eqn:mish2}
 \ell_{\pi}=\sum_{i,j\in I} e_{ij}\otimes \chi_i\chi_j\otimes \pi(s_{ij})\in M_{m}(\CCC)\otimes C(M)\otimes \mathcal{L}(V).
 \end{equation}

\end{notation}

\section{Almost flat bundles}\label{sect:3}
The goal of this section is to present a proof of Theorem~\ref{thm:aprox-monodromy} on approximate monodromy correspondence.
The corresponding result was noted without proof by Skandalis in a remark on page 313 of\cite{MR1157846}. It generalizes a result from \cite{CGM:flat} according to which $\ep$-flat bundles on simply connected manifolds are trivial if $\ep$ is sufficiently small.
In \cite{BB}, we had used the content of Theorem~\ref{thm:aprox-monodromy} in conjunction with the Mishchenko-Fomenko index theorem in order to extend an index theorem of Connes Gromov and Moscovici from \cite{CGM:flat}.

Fix $F\subset \pi_1(M,p)$ and $m$  as in Notation~\ref{not:misi} and the map $s\mapsto \gamma_s$ as in Notation~\ref{def:cgm}.

\begin{definition}[\cite{CGM:flat}] \label{def:con-flat} A unitary connection $\nabla$ on a bundle $E\in \Bun_A^V(M)$  is called $\ep$-flat, $\ep>0$,  if its norm
\[\|R^{\nabla}\| = \sup_{p\in M}\{\|R_p^{\nabla}(X,Y)\|\colon \|X\wedge Y\|\leq 1,  X,Y\in TM_p\},\]
satisfies $\|R^{\nabla}\|<\ep$.
In this case, the couple $(E,\nabla)$ is called  $\ep$-flat.
\end{definition}
The topological counterpart of $\ep$-flatness is the following.
\begin{definition}\label{def-flat} Let  $\mathcal{U}=(U_i)_{i\in I}$ be an open cover of $M$. A bundle $E\in \Bun_A^V(M)$ is
called $(\mathcal{U},\varepsilon)$-flat if is represented by a cocycle $v_{ij}:U_i\cap U_j \to U(V)$ such that
$\|v_{ij}(p)-v_{ij}(q)\|<\varepsilon$ for
all $p,q\in U_i\cap U_j$ and all $i,j \in I$.
\end{definition}

In the sequel we will occasionally identify a bundle $E$ in $\Bun_A^V(M)$ with the corresponding bundle in $\Bun^{\mathcal{L}(V)}_{\mathcal{L}(V)}(M)$ constructed from the same cocycle $v_{ij}:U_i\cap U_j \to U(V)=U(\mathcal{L}(V))$. $A$ and $\mathcal{L}(V)$ are Morita equivalent as $V$ is a finitely generated Hilbert $A$-module.
In particular, for $A=M_r(\C)$, we can identify and go back-and-forth between rank-one bundles of Hilbert $M_r(\C)$-modules  and rank-$r$ hermitian complex vector bundles  constructed from the same transition functions.

We are going to explain how Atiyah's map $\pi\mapsto L_{\pi}$ can be extended to approximate group representations.
Let $M$, $\G=\pi_1(M,p)$, $(\chi_i)_{i\in I}$ with $|I|=m$ and $F=\{s_{ij}\}$ be as in Notation~\ref{not:misi}.
For $\ep>0$ and $V$ a projective Hilbert $A$-module, we define
\begin{equation*}
\mathrm{Rep}^V_{(F,\ep)}(\G)=
\{\rho:\G\to U(V):\|\rho(st)-\rho(s)\rho(t)\|<\ep, \rho(s^{-1})=\rho(s)^{*}\!, \,s,t\in F,\, \rho(1)=1.\}
\end{equation*}
\begin{definition}\label{atiyah2}
	(a) The map $\mathrm{Rep}^V_{(F,\frac{1}{5m^2})}(\G) \to \Bun_A^V(M)$,
		$$\rho\mapsto L_{\rho}$$
	 is defined as follows. Consider the selfadjoint element $x_\rho= \sum_{i,j\in I} e_{ij}\otimes \chi_i\chi_j\otimes \rho(s_{ij})$ of the $C^*$-algebra $M_{m}(\CCC)\otimes C(M)\otimes \mathcal{L}(V)$.
Since
\[x_\rho^2-x_\rho=\sum_{i,k} \left(\sum_j e_{ik}\otimes \chi_i\chi_k \chi_j^2 \otimes (\rho(s_{ij})\rho(s_{jk})-\rho(s_{ik}))\right)\]
we see that
 $\|x_\rho^2-x_\rho\|<1/5$ and hence
	  the spectrum of $x_\rho$  is contained in $[0,1/3)\cup (2/3,1]$. It follows by functional calculus that
	 \[\ell_{\rho}=\chi_{(\frac{2}{3},1]} (x_\rho)\]
	 is a selfadjoint projection in the same $C^*$-algebra such that $\|x_\rho-\ell_\rho\|<1/3$.
	 The Hilbert $A$-module bundle corresponding to $\ell_\rho$ is denoted $L_\rho$.
	
	(b)  Let us note that if $\varphi : C^*(\G) \to A$ is a unital completely positive map such that
	$\|\varphi(st)-\varphi(s)\varphi(t)\|<\frac{1}{5m^2}$ for all $s,t \in F,$ then the selfadjoint element $x_\varphi=\sum_{i,j\in I} e_{ij}\otimes \chi_i\chi_j\otimes \varphi(s_{ij})$ satisfies  $\|x_\varphi^2-x_\varphi\|<1/5$ so that we can define
	the projection $\ell_\varphi$ and the corresponding bundles $L_\varphi$ as above.
\end{definition}
The following Lemma shows that approximate representations in close proximity yield isomorphic bundles.
\begin{lemma}\label{lemma:close}
	If $\rho,\rho' \in \mathrm{Rep}^V_{(F,\frac{1}{5m^2})}(\G)$ and $ \sup\limits_{s\in F}\|\rho(s)-\rho'(s)\|<\frac{1}{3m^2}$ then $\|\ell_\rho-\ell_{\rho'}\|<1$ and hence $L_{\rho}\cong L_{\rho}.$
\end{lemma}
\begin{proof}
$\|\ell_\rho-\ell_{\rho'}\|\leq \|\ell_\rho-x_{\rho}\|+\|x_\rho-x_{\rho'}\|+\|x_{\rho'}-\ell_{\rho'}\| <\frac{1}{3}+\frac{1}{3}+\frac{1}{3}=1.$
\end{proof}
\begin{remark}\label{rem5}
The following version of Lemma~\ref{lemma:close} holds. Depending only on the data from Notation~\ref{not:misi}, there is $\ep>0$ such that for $\rho \in \mathrm{Rep}^V_{(F,\frac{1}{5m^2})}(\G)$ and  any homomorphism $\pi:\G \to GL(V)$ satisfying
\[ \sup\limits_{s\in F}\|\rho(s)-\pi(s)\|<\ep,\]
the idempotent $\ell_\pi$ defined by the equation \eqref{eqn:mish2} is sufficiently closed to the selfadjoint projection $\ell_\rho$ so that they are
conjugated by an invertible element.
In particular, they define the same class in $K^0(M)$. The idempotent $\ell_\pi$   is not necessarily selfadjoint since $\pi$ takes values in $GL(V)$ rather than $U(V)$.

\end{remark}
 It is now well understood that the monodromy correspondence described in Proposition~\ref{prop:koba} extends to an approximate monodromy correspondence for almost flat bundles. This idea which was introduced in \cite{CGM:flat} and was explored in detail in
\cite{Carrion-Dadarlat}, \cite{Hunger} and \cite{Kubota3}, it is central for our paper.
In the sequel we will use a  version of the approximate monodromy correspondence in a smooth setting.

We will assume that  $\G=\pi_1(M,p)$ does not have elements of order $2$. This restriction is not really necessary, but we make it in order to streamline some of the arguments.
In any case, our main application involves only torsion free groups.
\begin{notation}\label{def:cgm}
	For each $s\in \G=\pi_1(M,p)$ choose a piecewise smooth loop $\gamma_s$ that represents $s$ with the provision that $\gamma_{s^{-1}}(\tau)=\gamma_{s}(1-\tau)$, $\tau\in [0,1]$.
This defines a map $\gamma:\G \to \Omega_1(M,p).$
With $\gamma$ fixed as above, for each a smooth hermitian vector bundle $E$ on $M$ endowed with
	 a unitary connection $\nabla$, we consider the map
	 $$\rho=\rho_{(E,\nabla)}:\G=\pi_1(M,p)\to U(E_p)$$
 $\rho(s)=h(\gamma_s)=P_{\gamma_s},$ $s\in \G$, where $P=P^{\nabla}$ is the parallel transport
	 in $E$  defined by $\nabla$.
\end{notation}

Consider two smooth curves $f_0,f_1:[0,1]\to M$ from $p$ to $q$ and let $(f_t)$, $0\leq t \leq 1,$ be a smooth homotopy between $f_0$ and $f_1$ with fixed endpoints.
Let $P_{f_i}:E_p\to E_q$ be the parallel transport along $f_i$.
Assume for the area of the homotopy $(f_t)$ that:
\[\iint {\|}{\partial_t}f_t(s)\wedge{\partial _s} f_t(s){\|}\,ds\,dt\leq C.\]

\begin{proposition}[Buser-Harcher]\label{BH}
$$\|P_{f_0}-P_{f_1}\|\leq \iint {\|}{R^\nabla(\partial_t}f_t(s),{\partial_s} f_t(s)){\|}\,ds\,dt\leq \|R^{\nabla}\|\iint {\|}{\partial_t}f_t(s)\wedge{\partial _s} f_t(s){\|}\,ds\,dt\leq \|R^{\nabla}\|C.$$
\end{proposition}
\begin{proof} This is proved in \cite[6.2.1]{Buser-Karcher}, see also \cite[Prop.2.7]{Hunger}.
\end{proof}
It is routine to extend Proposition~\ref{BH} to
piecewise smooth curves.
\begin{lemma}[\cite{CGM:flat}]\label{lemma:cgm}
	For $M$ a compact connected Riemannian manifold and $F$ a finite subset of $\G=\pi_1(M,p)$, there is a constant $C$ that depends only on $M$ and $F$ such that for any couple $(E,\nabla),$	\[\|\rho(st)-\rho(s)\rho(t)\|\leq C \|R^\nabla\|, \quad \text{for all}\quad s,t \in F.\]
\end{lemma}
\begin{proof}
For $s,t \in F,$ fix a homotopy between $\gamma_t*\gamma_s$ and $\gamma_{st}.$ Since $F$ is finite, there is a constant $C$ larger than the areas of all these homotopies. It follows then by  Proposition~\ref{BH} that $\|P_{\gamma_t*\gamma_s}-P_{\gamma_{st}}\|\leq C\|R^{\nabla}\|,$
for all $s,t\in F$. Since $\rho(s)\rho(t)=P_{\gamma_s}\circ P_{\gamma_t}=P_{\gamma_t*\gamma_s}$, this completes the proof.
\end{proof}
The following Proposition collects several facts from papers of Hanke and Schick~\cite{Hanke-Schick} and Hanke \cite{Hanke}.
\begin{proposition}[\cite{Hanke}]\label{prop:Hanke}
	Let $M$ be  a compact connected Riemannian manifold and let $\ep>0$. Let $(E_n,\nabla_n)$ be a sequence with $E_n\in \Bun^{r_n}_{\C}(M)$ and $\|R^{\nabla_n}||\leq \ep$ for all $n$.  Let $A_n=M_{r_n}(\C)$, and
	set  $A=\prod_n A_n$  and $B=\prod_n A_n/\bigoplus_n A_n$. Then
	\begin{itemize}
\item[(i)]
There is $E_A\in \Bun_A(M)$  with transition functions in diagonal form and such that the $n^{th}-$component of $E_A$ is isomorphic to $E_n$ as  Hilbert $A_n$-modules bundles.
\item[(ii)] The holonomy representations $h_n:\mathcal{P}_1(M)\to\mathcal{T}(E_n)$ defined via parallel transport  for  $(E_n,\nabla_n)$ assemble to a holonomy representation $h_A: \mathcal{P}_1(M)\to\mathcal{T}(E_A)$.
	\item[(iii)] 	
	If $\,\lim_n \|R^{\nabla_n}\|=0$,
	then the composition
	%\begin{equation*}
	$\label{help}\xymatrix{
h_B:\mathcal{P}_1(M)\ar[r]^-{h_A}& \mathcal{T}(E_A)\ar[r]&\mathcal{T}(E_A\otimes_A B)
}	$
%\end{equation*}
descends to a morphism of groupoids $\Pi_1(M)\to \mathcal{T}(E_A\otimes_A B)$ and hence induces a group homomorphism $\rho_B:\pi_1(M,p)\to U(B)$.  The Hilbert $B$-module bundle $E_A \otimes_{A} B$ is isomorphic to $L_{\rho_B}$.% where we denoted%he restriction of $\rho_B$ to $\G=\pi_1(M,p)$ by the same symbol.
 \end{itemize}	
\end{proposition}
	\begin{proof} $B$ is viewed as a left $A$ module via the  quotient map $q:A \to B$.
	In the notation of Kasparov, $E_A\otimes_A B=E_A\otimes_q B$.
		The statement of the proposition is the combination of Propositions 3.4, 3.12 and 3.13 from \cite{Hanke}.
		The proof of the these properties  is based on the following fact established in Proposition 3.4  of \cite{Hanke}. There are constants $C,\lambda>0$ depending only on $M$ such that
		for any couple $(E,\nabla)$ one has
$$\|P_\gamma-\id{E_p}\|\leq C \|R^\nabla\|\cdot\mathrm{length}(\gamma)$$ for each closed loop $\gamma\in \Omega_1(M,p)$ with $\mathrm{length}(\gamma)\leq \lambda$. The above estimate or
Lemma~\ref{lemma:cgm} also explain why
		the map $h_B$ descends to $\pi_1(M,p)$, when $\lim_n \|R^{\nabla_n}\|=0$.
	\end{proof}

Fix $F\subset \pi_1(M,p)$ and $m$  as in Notation~\ref{not:misi} and the map $s\mapsto \gamma_s$ as in Notation~\ref{def:cgm}.
\begin{theorem}\label{thm:aprox-monodromy}
	There is $\ep=\ep_M>0$ such that for any smooth hermitian vector bundle $E$ on $M$
	 which admits a unitary connection $\nabla$ with
	 $\|R^\nabla\|<\ep$, the map $\rho(s)=P^\nabla_{\gamma_s},$ $s\in \G=\pi_1(M,p),$ defined by parallel transport
	 is an approximate representation $\rho\in \mathrm{Rep}^r_{(F,\frac{1}{5m^2})}(\G)$, $r=\mathrm{rank}(E)$,
	   with the property that
	$L_{\rho}\cong E$.
\end{theorem}
\begin{proof}
By Lemma~\ref{lemma:cgm}, there is $\ep_0$ such that if $\|R^\nabla\|<\ep_0$,
then $\rho\in \mathrm{Rep}^r_{(F,\frac{1}{5m^2})}(\G)$.
Seeking a contradiction, suppose that the statement is false for all  $\ep>0.$ Consequently, if we fix a sequence $(\ep_n)$ convergent to $0$ with $\ep_n<\ep_0,$ then there is a sequence $(E_n,\nabla_n)$ such that $E_n\in \Bun_{\C}^{r_n}(M)$,  $\|R^{\nabla_n}\|<\ep_n$
but $E_n \ncong L_{\rho_n}$ for all $n\in N$.

Let $E_A$, $h_A,$ $h_B$ and $\rho_B$ be as in Proposition~\ref{prop:Hanke} so that in particular  $E_A \otimes_{A} B\cong L_{\rho_B}$. Define $\rho_A: \G \to U((E_A)_p)\cong U(A)$ by $\rho_A(s)=h_A(\gamma_s).$ Thus the components of $\rho_A$ are $(\rho_n)$ and
  $q\circ \rho_A=\rho_B$, where $q:A \to B$ is the quotient map.
Consider the element
$$x_{\rho_A}=
\sum_{i,j\in I} e_{ij}\otimes \chi_i\chi_j\otimes \rho_A(s_{ij})\in M_m(\C)\otimes C(M)\otimes A,$$
and let $\ell_{\rho_A}=\chi_{(2/3,1]}(x_{\rho_A})$ be the projection representing $L_{\rho_A}$ as Definition \ref{atiyah2}.
 Then $(\id{M_m(\C)\otimes C(M)}\otimes q )(\ell_{\rho_A})=\ell_{\rho_B}$ is a projection representing $L_{\rho_B}.$ Using Proposition~\ref{prop:Hanke}(iii), it follows that
 $$L_{\rho_A}\otimes_A B \cong L_{\rho_B}\cong E_A \otimes_{A} B.$$
 The components of $\ell_{\rho_A}$ are $(\ell_{\rho_n})$ by naturality of functional calculus.

 This implies that $L_{\rho_n}\cong (L_{\rho_A})_n\cong (E_A)_n \cong E_n,$ for all sufficiently large $n$,
 and this contradicts our assumption.
\end{proof}

\section{Uniform-to-local non-stable groups}\label{sect:4}

Let $M$ be a closed connected Riemannian manifold with nonpositive sectional curvature,  $K(M)\leq 0$.
Fix a base point $p\in M$ and
consider  the fundamental group $\Gamma=\pi_1(M,p)$. It is a classic result that
$\G$ is torsion free and that each homotopy class $s\in \pi_1(M,p)$ contains a unique constant speed geodesic  $\gamma_s:[0,1]\to M$, see \cite[Thm.4.13, p.200]{Bridson-book}. In particular
 $\gamma_{s^{-1}}$ is  $\gamma_s$ with the opposite parametrization, $\gamma_{s^{-1}}(\tau)=\gamma_s(1-\tau)$.
Thus, for manifolds with nonpositive sectional curvature we have a preferred choice for the map
  $\G=\pi_1(M,p)\to \Omega_1(M,p), $  $s\mapsto \gamma_s,$ considered in Notation~\ref{def:cgm}.
  For an $\ep$-flat couple $(E,\nabla)$ as in Definition~\ref{def:con-flat},
 parallel transport along these geodesics define a map
  $\rho:\Gamma\to U(E_p)$,
  \begin{equation}\label{rho}
   \rho(s)=P_{\gamma_s},\,\, s\in \Gamma.
  \end{equation}
  with $\rho(s^{-1})=\rho(s)^{-1}$ for all $s\in \G$. The map $\rho$ is an approximate unitary representation by Lemma~\ref{lemma:cgm}.
  Morever, as indicated by Gromov
on page 166 of \cite{Gromov-large}, $\rho$ is a uniform approximate representation if $M$ has  strictly negative sectional curvature.
\begin{proposition} [Gromov] \label{Gro}Let $M$ be a closed connected Riemannian manifold with sectional curvature
$K(M)\leq -\kappa <0.$  Then $\|\rho(st)-\rho(s)\rho(t)\|\leq \kappa \pi \|R^{\nabla}\|$ for all $s,t \in \G$.
\end{proposition}
\begin{proof}
Consider  the universal cover $\Phi:\widetilde{M} \to M$.  We endowed $\widetilde{M}$ with the pullback of the Riemannian metric from $M$ so that $\Phi$ becomes a local isometry. Fix a base point $\tilde p\in \widetilde{M}$ with $\phi(\tilde p)=p$.
$\G$ acts by isometries on  $\widetilde{M}$:  $\G\times \widetilde{M}\to \widetilde{M}$, $(t,\tilde x) \mapsto t\cdot \tilde x$. For each $t\in G$, let $\tilde{\gamma}_t:[0,1]\to \widetilde{M}$ be the unique geodesic  joining $\tilde p$ with
$t\cdot \tilde p$. Then $\Phi\circ \tilde{\gamma}_t={\gamma}_t$ by the uniqueness result mentioned earlier, see \cite[p.201]{Bridson-book}. Fix $s$ and $t$ in $\G$. Let $\bar{\gamma}_s$ be the unique lift of $s$ to a path in $\widetilde{M}$ with $\bar{\gamma}_s(0)=t\cdot \tilde p$. Then $\gamma_t*\bar{\gamma}_s $ lifts $st$.
Consider the  geodesic triangle in $\widetilde{M}$ with vertices $\tilde{p}$,  $t\cdot\tilde{p}$, and $(st)\cdot\tilde{p}$ and  edges
$\tilde{\gamma}_t$, $\bar\gamma_s$ and $\tilde\gamma_{st}$. Following \cite[p.106]{Buser-Karcher}, we span a ruled surface $S$ into this geodesic triangle by joining  $\tilde{\gamma}_t(\tau)$
to $\tilde{\gamma}_{st}(\tau)$ by the unique minimizing geodesic for each $\tau\in [0,1]$ (geodesics vary continuously and smoothly on their endpoints).
From Gauss' equation the intrinsic curvature of $S$ is $\leq K(M)$ and by A.D. Aleksandrov's
area comparison theorem \cite[Prop.6.7]{Buser-Karcher},  the area of $S$ is less than equal to the area of the model triangle $T$ with the same edge length.  Since   the model space has constant curvature $-\kappa<0$, it follows that $\textrm{area}(T)\leq \kappa \pi$ by Gauss-Bonnet,  and  hence  $\textrm{area}(S)\leq \kappa \pi.$
This shows that there is a homotopy $(\tilde{f}_\tau)$ of area $\leq \kappa \pi$ between the paths and $\tilde{\gamma}_t *(\bar\gamma_s)$ and $\tilde\gamma_{st}$.
Since $\Phi$ is a local isometry, it follows that $f_\tau:=\Phi\circ \tilde{f}_\tau$ is a homotopy of area $\leq \kappa \pi$
between $f_0=\gamma_t*\gamma_s$ and $f_1=\gamma_{st}.$ It follows by Proposition~\ref{BH} that
\[ \|P_{\gamma_s}\circ P_{\gamma_t}-P_{\gamma_{st}}\|=\|P_{\gamma_t*\gamma_s}-P_{\gamma_{st}}\|\leq \|R^{\nabla}\|\kappa \pi .\]
\end{proof}
The dual assembly map of Kasparov \(\nu:K^0(C^*(\G))\to  RK^0(B\G)\) is a generalization of the Atiyah's map $\mathrm{Rep}(\G) \to RK^0(B\G)$, $\pi \mapsto [L_\pi]$.
Kasparov has shown that if  $\G$ has a $\gamma$-element, then $\nu$ is surjective, \cite{Kas:inv}.
For the proof of Theorem~\ref{thm:main-resultA} we only use the surjectivity of $\nu$ in the case when $\G$ is the fundamental group of a closed Riemannian manifold with nonpositive sectional curvature. This was proven by Kasparov  in \cite[Thm.6.7]{Kas:inv} by showing that $\G$ has a $\gamma$-element. By the Hadamard-Cartan theorem the universal cover of $M$ is contractible and hence $M$ is a model for $B\G$.
Moreover, it turns out that for quasidiagonal groups (e.g. residually finite groups), the map $\nu$ remains surjective even if one restricts its domain to quasidiagonal K-homology classes $K^0(C^*(\G))_{qd}$,  \cite{AA}, \cite{Kubota2}, \cite{CCC}. We will elaborate on this point in the proof of Theorem~\ref{thm:main-resultA} below.

\begin{theorem}\label{thm:main-resultA}
  Let $M$ be a closed connected Riemannian manifold with sectional curvature
$K(M)\leq -\kappa <0.$ Assume that  $\Gamma=\pi_1(M)$ is residually finite,  and $b_{2i}(M)\neq 0$ for some $i>0$.  Then there exist a finite subset $F\subset \G$ and $C>0$  with the following property. For any $\ep>0$ there is a  unital map $\rho:\G \to U(n)$, satisfying
$\sup_{s,t\in \G}\|\rho(st)-\rho(s)\rho(t)\|<\ep,$
 such that for any representation $\pi:\G\to GL_n(\C),$\,
 \(\sup_{s\in F} \|\rho(s)-\pi(s)\|>C.\)
\end{theorem}
\begin{proof} Let $F\subset \G$ be the finite set from Notation~\ref{not:misi}.
We construct a sequence of unital maps
 $\{\rho_n:\G \to U(H_n))\}_n$ with $H_n$ finite dimensional Hilbert spaces and
 \[\lim_{n\to \infty} \,\sup_{s,t\in \G}\|\rho(st)-\rho_n(s)\rho_n(t)\|=0,\]
 for which does not  exist a sequence of homomorphisms $\{\pi_n:\G\to GL(H_n)\}_n$
 such that
 \begin{equation}\label{FFF}
 \lim_{n\to \infty} \|\rho_n(s)-\pi_n(s)\|=0, \,\text{for all} \, s\in F.
\end{equation}	

Since $M=B\G$ and $b_{2k}(\G)\neq 0$, one has $H^{2k}(M,\Q)\neq 0$.  Since the Chern character is a rational isomorphism, there is
  a nontorsion element $y\in \tilde{K}^0(M).$
 Since  the dual assembly map $\nu: K^0(C^*(\G))_{qd}\to K^0(B\G)=K^0(M)$ is surjective by  \cite[Thm.4.6]{CCC}, there is $x\in K^0(C^*(\G))_{qd}$ such that $\nu(x)=y$.
 The following realization of $\nu$ on quasidiagonal KK-classes was introduced in  \cite{AA}.
The element  $x$ is represented  by  a pair of unital $*$-representations $\Phi,\Psi:C^*(\G)\to \mathcal{L}(H)$,  such that $\Phi(a)-\Psi(a)\in K(H)$,
 $a\in C^*(\G)$, and with property that
 there is an increasing approximate unit $(p_n)_n$ of $K(H)$ consisting of projections such that $(p_n)_n$ commutes asymptotically with both $\Phi(a)$ and $\Psi(a)$, for all $a\in C^*(\G)$. Thus:
 \begin{equation}\label{qddd}
 	\lim_n \|p_n\Phi(a)-\Phi(a)p_n\|=0, \quad \lim_n \|p_n\Psi(a)-\Psi(a)p_n\|=0, \quad a \in C^*(\G).
 \end{equation}
 Moreover, since the group $\G$ is residually finite, there is a
 residually finite dimensional $C^*$-algebra $D$ which is intermediate between $C^*(\G)$ and $C_r^*(\G)$
 so that $x$ is in the image of the map $K^0(D)\to K^0(C^*(\G))$. This is explained in \cite{CCC} (Prop. 3.8 and Thm.4.6)  where we used Kubota's idea \cite{Kubota2} of considering quasidiagonal $C^*$-algebras which are intermediate between  the full and the reduced group $C^*$-algebras.
 This fact allows us to assume that $\Psi$ is given by a direct sum of finite dimensional unitary representations of $\G,$ so that we may arrange that each $p_n$ commutes with $\Psi$.

 If $r_n=\mathrm{rank}(p_n)$,  set
  $A_n=M_{r_n}(\C)$,
	 $A=\prod_n A_n$,  $B=\prod_n A_n/\bigoplus_n A_n$, and let $q:A \to B$ be the quotient map.
Define unital  maps $\varphi,\psi: C^*(\G) \to A$ by
 \[\varphi(a)=(\varphi_n(a))_n \quad\text{and} \quad \psi(a)=(\psi_n(a))_n ,\]
 where $\varphi_n(a)=p_n \Phi(a) p_n$ and $\psi_n(a)=p_n \Psi(a) p_n.$ Then $\varphi$ is a unital completely positive map with asymptotically  multiplicative components and $\psi$ is a unital $*$-homomorphism.
%
% \MC{Let  $F=\{s_{ij}\subset \G\}$ be as in Notation~\ref{not:misi}.
% Choose  $k\in \N$ such that $\sigma_n(s)=\varphi_n(s)\varphi_n(s)^{-1/2}\in U(A)$ for all  $s\in F\cdot F$ and $n\geq k$. Extend each $\sigma_n$ to a map $\G\to U(A_n)$ by setting $\sigma_n(s)=1$
% if either $n<k$ or $s \notin F\cdot F$.
%It follows that $\lim_n \|\varphi_n(s)-\sigma_n(s)\|=0$ for all $s\in F\cdot F$.}
%
 Let $L_\varphi$ and $L_\psi$ be the corresponding Hilbert $A$-module bundles, with components $L_{\varphi_{n}}$ and $L_{\psi_n}$, see Definition~\ref{atiyah2}.
 As we noted in \cite{AA} based on work of Phillips and Stone \cite[3.18]{MR832541}, \cite[Thm.1]{MR1124246},  the bundles  $L_{\varphi_{n}}$  are $(\mathcal{U},\ep'_n)$-flat
  with  $\lim \ep'_n =0$ and each $L_{\psi_n}$  is flat since $\psi_n$ is a true representation. Moreover, by \cite{AA}:
  \[y=\nu(x)=[L_{\varphi_{n}}]-[L_{\psi_n}],\quad  \text{for all}\,\, n \geq 1.\]
  The definition of  almost flat bundles based on smooth connections as defined in \cite{CGM:flat}, see Definition~\ref{def:con-flat}, is equivalent with the  definition based on flatness of transition functions, see Definition~\ref{def-flat}, as verified by Hunger \cite{Hunger}. It follows that one can replace the bundles $L_{\varphi_{n}}$ by smooth bundles $E_n$ isomorphic to $L_{\varphi_{n}}$ and which are
  endowed with metric compatible connections $(E_n,\nabla_n)$  such that $\|R^{\nabla_n}\|\leq \ep_n$ where $\ep_n \leq C' \ep_n'$  and hence  $\lim \ep_n =0$. The constant $C'$  depends only on $M$ and the fixed cover $\mathcal{U}$.
  Let $h_n:\mathcal{P}_1(M)\to\mathcal{T}(E_n)$ be the holonomy representations  defined via parallel transport  corresponding to $(E_n,\nabla_n)$ and let $\rho_n(s)=h_n(s)=P^{\nabla_n}_{\gamma_s}$, $s\in \G$.
  It follows from Theorem~\ref{thm:aprox-monodromy} that
  \[ L_{\rho_n}\cong E_n \cong L_{\varphi_{n}} \]
  for all sufficiently large $n$. On the other hand, since we work with geodesic loops $\gamma_s$ representing the elements of $\Gamma$,
it follows by Proposition~\ref{Gro}  that $$\sup_{s,t\in \G}\|\rho_n(st)-\rho_n(s)\rho_n(t)\| \leq \kappa \pi \ep_n.$$
     We claim that the sequence $(\rho_n)$ cannot be perturbed  to a sequence  of representations satisfying ~\eqref{FFF}. Seeking a contradiction, assume that there exists a sequence of representations $\pi_n:\G \to GL_{r_n}(\C)$ such that $\lim_n \|\rho_n(s)-\pi_n(s)\|=0$ for all $s\in F$. In this case, it follows from Remark~\ref{rem5} that
      $L_{\rho_n} \cong L_{\pi_n}$ for all sufficiently large $n$ and hence
      $y=[L_{\rho_n}]-[L_{\psi_n}]=[L_{\pi_{n}}]-[L_{\psi_n}]\in \widetilde{K}^0(M)$. By \cite{Milnor}, all Chern classes of a flat complex bundle vanish over $\Q$, so that the Chern character of $y$ is zero as both $\pi_n$ and $\psi_n$ are representations.
     We obtained a contradiction, since  $y$ is a nontorsion element of $\widetilde{K}^0(M)$. \end{proof}
     \begin{remark} There is a related approach to Theorem~\ref{thm:main-result} based on the notion K-area of a Riemannian manifold introduced by Gromov \cite[\S 4]{Gromov:curvature}.  Indeed, if $M$ is an orientable closed connected Riemannian manifold with non-positive sectional curvature of dimension  $\dim (M)=2m$ and residually finite fundamental group, then $M$ has infinite K-area  by \cite[Ex.(v') p.25]{Gromov:curvature}, see also ~\cite{Hanke-Schick}.
     This yields a sequence $(E_n,\nabla_n)$ with $\lim_n \|R^{\nabla_n}\|=0$ and  top Chern class $c_m(E_n)\neq 0$. Having this sequence at hand, one proceeds like in the second part of the proof of Theorem~\ref{thm:main-result}.
     \end{remark}

\section{Obstructions to $C^*$-stability}\label{sec:semi}

In this section we establish Theorem~\ref{thm-cstable} as a consequence of a stronger result which assumes weaker forms of stability,
 see Theorem~\ref{thm:2nd} and Theorem~\ref{cor2nd}.

 We refer the reader to \cite{Kas:inv} for the definitions and the basic properties of the various KK-theory groups introduced by Kasparov. We will freely employ the same notation as  there.
Thus if $X$ is a locally compact group and $A$, $B$ are separable $C^*$-algebras we write
$RKK^0(X;A,B)$ for ${\CR}KK(X;C_0(X)\otimes A,C_0(X)\otimes B)$ and
$RK^0(X;B)$   for $RKK^0(X;\C,B)$.

Let $\underline{E}\G$ be the classifying space for proper actions of $\G$, \cite{BCH}. It is known that $\underline{E}\G$ admits a locally compact model, \cite{Kasparov-Skandalis-Annals}.
%and $\sigma$-compact model.
Let us recall  that $\G$ has a $\gamma$-element if there exists a $\G-C_0(\underline{E}\G)$-algebra $A$ in the sense of \cite{Kas:inv} and two elements $d\in KK_G(A,\C)$ and $\eta\in KK_{\G}(\C,A)$ (called Dirac and dual-Dirac elements, respectively) such that the element $\gamma=\eta\otimes_Ad\in KK_\G(\C,\C)$ has the property that $p^*(\gamma)=1\in \mathcal{R}KK_\G(\underline{E}\G;C_0(\underline{E}\G),C_0(\underline{E}\G))$ where
$p:\underline{E}\G \to \text{point}$, \cite{Tu:gamma}.
Let $B$ be a separable $C^*$-algebra endowed with the trivial $\G$-action.
Consider the dual assembly map with coefficients in $B$:
$$\alpha: KK_{{\G}}(\C,B) \to RKK_{\G}^0(\underline{E}{\G};\C,B)$$
defined by $\alpha(y)=p^*(y)$ where $p:\underline{E}{\G} \to \text{point}$.
As in \cite{Kas:inv}, we write $RKK_{\G}^0(\underline{E}{\G};\C,B)$ for ${\CR}KK_{\G}(\underline{E}{\G};C_0(\underline{E}{\G}),C_0(\underline{E}{\G})\otimes B).$
.

We need the following result which is essentially due to Kasparov, \cite[Th.6.5]{Kas:inv}. It was discussed in  \cite[Thm.7.1]{Meyer-Nest:BC}, \cite[Thm.23]{Emerson-Meyer} and \cite[Lem.10.1]{Emerson-Nica}.
\begin{theorem}[Kasparov]\label{thm:Kas1}
  If ${\G}$ is a countable discrete group that admits a $\gamma$-element,
  then the dual assembly map $\alpha:KK_{\G}(\C,B)\to {R}KK_{\G}^0(\underline{E}{\G};\C,B)$ is split surjective with kernel \mbox{$(1-\gamma)KK_{\G}(\C,B)$}.
\end{theorem}

  By universality of $\underline{E}{\G}$, there is a ${\G}$-equivariant map (unique up to homotopy) $\sigma: E{\G} \to \underline{E}{\G}$.
  It induces a map $\sigma^*:{R}KK_{\G}^0(\underline{E}{\G};\C,B)\to {R}KK_{\G}^0({E}{\G};\C,B)$. Recall that $\QQ=\bigotimes_n M_n$  is universal UHF algebra and $K_0(\mathcal{Q})=\Q$ and $K_1(\mathcal{Q})=0$. We can identify
the representable rational K-theory group $RK^0(X;\Q)$ with $RK^0(X;\QQ)$, as in \cite{Kasparov-conspectus}. We view $\mathcal{Q}$ as a trivial ${\G}$-algebra.
\begin{corollary}\label{cor:KY}
  Let ${\G}$ be a countable discrete group that admits a $\gamma$-element.
  Let $B$ be a separable trivial ${\G}$-algebra  such that $B\cong B \otimes \mathcal{Q}$.
  Then the composition $$\gamma KK_{\G}(\C,B)\hookrightarrow KK_{\G}(\C,B)\stackrel{\alpha}\longrightarrow {R}KK_{\G}^0(\underline{E}{\G};\C,B)\stackrel{\sigma^*}\longrightarrow {R}KK_{\G}^0({E}{\G};\C,B)$$
   is  a surjective map.
\end{corollary}
\begin{proof}
 It was shown in \cite[p.275-6]{BCH} that $\sigma$ induces a rationally injective homomorphism
 \[\sigma_*: RK^{\G}_0(E{\G}) \to RK^{\G}_0(\underline{E}{\G}).\]
   It follows that the map
  \((\sigma_*)^*:\Hom(RK^{\G}_0(\underline{E}{\G}),K_0(B))\to \Hom( RK^{\G}_0(E{\G}),K_0(B))\) is surjective since $K_0(B)\cong K_0(B)\otimes \Q$.
  By the universal coefficient theorems  stated as Lemma 2.3 of \cite{Kasparov-Skandalis-cr}  and Lemma 3.4 of \cite{Kasparov-Skandalis-kk} applied for the coefficient algebra $B$, the horizontal maps in the commutative diagram
   \[\xymatrix{
{R}KK_{\G}^0(\underline{E}{\G};\C,B)\ar[d]_{\sigma^*}\ar[r]& \Hom(RK^{\G}_0(\underline{E}{\G}),K_0(B))\ar[d]^{(\sigma_*)^*}
\\
{R}KK_{\G}^0({E}{\G};\C,B)\ar[r]  & \Hom( RK^{\G}_0(E{\G}),K_0(B))
}
\]
  are bijections.
  It follows that the restriction map
 $\sigma^*:{R}KK_{\G}^0(\underline{E}{\G};\C,B)\to {R}KK_{\G}^0({E}{\G};\C,B)$ is surjective. The  statement follows now from Theorem~\ref{thm:Kas1}.
\end{proof}

Let us recall that a set of operators $S\subset \mathcal{L}(H)$ on a separable Hilbert space $H$ is quasidiagonal if there exists an approximate unit of projections $(p_n)_n$ of  $K(H)$ such that $$\lim_{n\to \infty}\|[a,p_n]\|= 0,\quad \text{for all}\,\, a\in S.$$
A representation   $\pi:A \to \mathcal{L}(H)$ of a separable $C^*$-algebra $A$  is quasidiagonal if the set $\pi(A)$ is quasidiagonal.
 A $C^*$-algebras $A$ is quasidiagonal if it admits a faithful quasidiagonal representation.

Let $A$, $B$ be separable $C^*$-algebras.
Any class $x\in KK(A,B)$ is  represented by some Cuntz pair, i.e. a pair
   of $*$-homomorphisms $\varphi,\psi:A \to M(K(H)\otimes B)$,  such that $\varphi(a)-\psi(a)\in K(H)\otimes B$, for all $a\in A$. Assume that $B$ is unital.
\begin{definition}[\cite{AA},\cite{CCC}]\label{def:kkqd}
An element $x\in KK(A,B)$ is quasidiagonal if it is represented by a Cuntz pair  $(\varphi,\psi)$  with the property that
there exists an approximate unit of projections $(p_n)_n$ of  $K(H)$ such that $\lim_{n\to \infty}\|[\psi(a),p_n\otimes 1_B]\|= 0$,  for all $a\in A$.
The quasidiagonal elements form a subgroup of $KK(A,B)$,  denoted by $KK(A,B)_{qd}$. %The $C^*$-algebra $A$ is called K-quasidiagonal if $K^0(A)=K^0(A)_{qd}$.
For other contexts it will be useful to modify the definition by asking for the existence of an approximate unit of $K(H)\otimes B$ consisting of projections $(q_n)$ such that $\lim_{n\to \infty}\|[\psi(a),q_n]\|= 0$,  for all $a\in A$.
  \end{definition}
  \begin{remark}\label{remark:qdd}
  (a) If $\theta:A \to D$ is a $*$-homomorphism, then $\theta^*[\varphi,\psi]= [\varphi\circ \theta,\psi\circ \theta]$ and hence $$\theta^*(KK(D,B)_{qd})\subset KK(A,B)_{qd}.$$

    (b) Let $D, B$ be  separable unital $C^*$-algebras with $B$ nuclear. Fix a faithful unital representation $\psi_0:D \to M(K(H))$ such that $\psi_0(D)\cap K(H)=\{0\}.$  Then any element $x\in KK(D,B)$ is represented by a Cuntz pair $(\varphi,\psi)$ where $\psi=\psi_0 \otimes 1_B: A \to M(K(H)\otimes B)$, \cite{Skandalis:K-nuclear}. Therefore, if  $D$ is quasidiagonal,
     then $\psi_0(D)$ is a quasidiagonal subset of $M(K(H))$ and hence
      $KK(D,B)=KK(D,B)_{qd}$.
\end{remark}

Let us recall that  a countable discrete group $\G$ is \emph{quasidiagonal} if there is a faithful unitary representation $\pi:\G \to U(H)$ such that the set $\pi(\G)$ is quasidiagonal. Note that residually finite groups or more generally maximally periodic groups (abbreviated MAP) are quasidiagonal. Residually amenable groups are quasidiagonal by \cite{TWW}.
Quasidiagonality of a group $\G$ is at least formally weaker than  quasidiagonality of the full group $C^*$-algebra $C^*(\G)$. For example, the question of quasidiagonality of $C^*(\G)$ is completely open
for all residually finite,  infinite, property T groups.

\begin{proposition}[\cite{CCC}]\label{cor:qd-embed} %Let $\G$ be a countable discrete group. The following assertions are equivalent:
For   a countable discrete group $\G$, the following assertions are equivalent:
\begin{itemize}
  \item[(i)] $\G$ is quasidiagonal.
  \item[(ii)] $\lambda_{\G}$ is weakly contained in a quasidiagonal representation $\pi$ of $\G$.
  \item[(iii)] The canonical  map $q_{\G}:C^*(\G)\to C_r^*(\G)$ factors through  a unital  quasidiagonal C*-algebra.
\end{itemize}
\end{proposition}

Let $j_\G$ and $j_{\G,r}$ be the descent maps of Kasparov \cite[Thm.3.11]{Kas:inv}.  Thus $\gamma\in KK_\G(\C,\C)$ gives
an element $j_\G(\gamma)\in KK(C^*(\G),C^*(\G))$ which induces a map $$j_\G(\gamma)^*=j_\G(\gamma) \otimes_{C^*(\G)}-:KK(C^*(\G),B) \to KK(C^*(\G),B).$$
 The image of $j_\G(\gamma)^*$ is usually denoted by $\gamma KK(C^*(G),B)$,
while $\gamma_r KK(C_r^*(\G),B)$ is defined similarly as the image of $ j_{\G,r}(\gamma)^*$.
 Since $G$ is discrete and acts trivially on $B$,
there is a canonical  isomorphism (called dual Green-Julg isomorphism), \cite{Kas:inv},  $$\kappa: KK_\G(\C,B)\stackrel{\cong}\longrightarrow KK(C^*(\G),B)$$  which is compatible with  the module structure over the group ring of $\G$. Moreover, by \cite[Lemma 11]{Proietti},
 for every $x\in KK_\G(\C,\C)$, the following diagram is commutative.
 \begin{equation}\label{eq:obs}
\xymatrix{
KK_\G(\C,B)\ar[r]^-{\kappa}\ar[d]^{x\otimes -}&  KK(C^*(\G),B)\ar[d]^{j_\G(x)\otimes -}
\\
KK_\G(\C,B)\ar[r]_-{\kappa}  & KK(C^*(\G),B)
}
\end{equation}
Kasparov \cite[3.12]{Kas:inv} has shown that
  the canonical surjection  $q_\G :C^*(\G)\to C_r^*(\G)$  induces an isomorphism of $\gamma$-parts
   $q_\G^*: \gamma_r KK(C_r^*(\G),B) \stackrel{\cong}\longrightarrow  \gamma  KK(C^*(\G),B)$. In particular:

\begin{proposition}[Kasparov \cite{Kas:inv}]\label{prop:KO}
 If $\G$ is a discrete countable group  that admits a $\gamma$-element, then
 $\gamma  KK(C^*(\G),B)\subset q_{\G}^* (KK(C_r^*(\G),B))$.
\end{proposition}
We gave an exposition of  Proposition~\ref{prop:KO} in  \cite{CCC}. By \cite[Thm.3.4]{Kas:inv}, see also \cite[p.313]{Kasparov-Skandalis-kk},   there is natural descent isomorphism
  $$ \lambda^{\G}:{R}KK_{\G}^0({E}\G;\C,B)\stackrel{\cong}\longrightarrow  {R}KK^0(B\G;\C,B).$$
 Let $\nu: KK^0(C^*({\G}),B) \to {R}K^0(B\G;B)$ be the map
\begin{equation*}
 \nu=\lambda^{\G}\circ\sigma^*\circ \alpha \circ \kappa^{-1} : KK(C^*({\G}),B)\cong KK_{\G}(\C,B) \to {R}KK_{\G}^0({E}\G;\C,B)\cong  {R}KK^0(B\G;\C,B).
\end{equation*} Note that $\sigma^*\circ \alpha =p^*$ where $p:E\G \to point$.
Abusing terminology, we shall also  refer to $\nu$ as the dual assembly map.

As in \cite{CCC}, we rely on Kubota's idea \cite{Kubota2} of using a quasidiagonal $C^*$-algebra intermediate between $C_r^*(\G)$ and $C^*(\G)$ which
 strengthens significantly the construction of  almost flat K-theory classes based on {$K$}-{quasidiagonality} of $C^*(\G)$,  introduced in \cite{AA}.
 
In particular Thm. 4.6 from \cite{CCC} admits a version with coefficients:
\begin{theorem}\label{cor:qd-embed-kk}
  Let $\G$ be a countable discrete quasidiagonal group and let $B$ be a separable nuclear unital $C^*$-algebra.
  If $\G$ admits a $\gamma$-element, then $ \gamma KK(C^*(\G),B) )\subset KK(C^*(\G),B)_{qd}$. It follows that
$\nu (KK^0(C^*(\G),B))= \nu(KK(C^*(\G),B)_{qd})$ and hence
  \(\nu(KK(C^*(\G),B)_{qd})=RK^0(B\G;B) \) if $B\cong B \otimes \mathcal{Q}$.
\end{theorem}
\begin{proof} The factorization
$C^*(\G) \stackrel{q_D}\longrightarrow D  \to C_r^*(\G)$ of $q_{\G}$
with $D$ unital and quasidiagonal given by Proposition~\ref{cor:qd-embed} in conjunction with Remark~\ref{remark:qdd} implies that
\[q_{\G}^* (KK(C_r^*(\G),B))\subset q_D^*(KK(D,B))= q_D^*(KK(D,B)_{qd})\subset KK(C^*(\G),B)_{qd}.\]
From this and Proposition~\ref{prop:KO} we obtain that $ \gamma KK(C^*(\G),B) )\subset KK(C^*(\G),B)_{qd}$.

By  Theorem~\ref{thm:Kas1},  $\alpha$ vanishes on $(1-\gamma) KK_{\G}(\C,B))$. Since the diagram \eqref{eq:obs} is commutative, this group is mapped to  $(1-\gamma) KK^0(C^*(\G),B)$ by $\kappa$   and hence
$$\nu (KK^0(C^*(\G),B))= \nu(\gamma KK^0(C^*(\G),B))= \nu(KK(C^*(\G),B)_{qd}). $$
   By  Corollary~\ref{cor:KY}, the map $\nu$ is surjective if  $B\cong B \otimes \mathcal{Q}$.
\end{proof}
\begin{definition}\label{def:B-stable}
	 Let  $\G$ be a countable discrete group and let $\mathcal{B}$ be a class of unital $C^*$-algebras.
	\begin{itemize}
\item[(a)]	 $\G$ is called locally $GL(\mathcal{B})$-stable if  for any sequence of unital maps
 $\{\varphi_n:\G \to U(B_n)\}$ with $B_n\in\mathcal{B}$  and
 \begin{equation}\label{a1}\lim_{n\to \infty} \|\varphi_n(st)-\varphi_n(s)\varphi_n(t)\|=0, \quad \text{for all}\,\, s,t\in \G,\end{equation}
 there exists a sequence of homomorphisms $\{\pi_n:\G\to GL(B_n)\}$ such that
\begin{equation}\label{a2}\lim_{n\to \infty}  \|\varphi_n(s)-\pi_n(s)\|=0, \quad \text{for all}\,\, s\in \G.\end{equation}
 \item[(b)] $\G$ is called locally $U(\mathcal{B})$-stable if for any sequence of unital maps
 $\{\varphi_n:\G \to U(B_n)\}$ that satisfies \eqref{a1}, there exists a sequence of homomorphisms $\{\pi_n:\G\to U(B_n)\}$  satisfying \eqref{a2}.
 \item[(c)] If $\mathcal{B}=\{B_n: n\geq 1\},$ property (a) will be also called  local $\{GL(B_n): n\geq 1\}$-stability and property (b)  local $\{U(B_n): n\geq 1\}$-stability
 \end{itemize}
\end{definition}
\begin{remark}\label{remm}
The following observations are immediate.
\begin{itemize}
\item[(i)] If $\G$ is locally $U(\mathcal{B})$-stable, then $\G$ is locally $GL(\mathcal{B})$-stable.
\item[(ii)] If $\mathcal{B}$ is the class of all unital separable $C^*$-algebras, then $\G$ is locally $U(\mathcal{B})$-stable if and only if $\G$ is $C^*$-stable in the sense \cite{ESS-published} or equivalently $C^*(\G)$ is weakly semiprojective.
\item[(iii)] If $\mathcal{B}=\{M_n(\C): n \in \mathbb{N}\}$, then $\G$ is locally $U(\mathcal{B})$-stable if and only if
$\G$ is matricially stable in the sense of \cite{ESS-published}, \cite{CCC}, or locally stable in the sense of Definition~\ref{def:main}.
\item[(iv)] Let  $X$ be a compact metrizable space.  It is easily verified  that if an MF group $\G$ is
 locally $\{GL_n(C(X)): n \in \mathbb{N}\}$-stable, then $\G$ is also locally $\{GL_n(\C): n \in \mathbb{N}\}$-stable.

%\item[(iv)] if $\mathcal{B}=\{M_n(C(\T)): n \in \mathbb{N}\}$.
\end{itemize}
\end{remark}

\begin{lemma}\label{lem:easy} Let $\G$ be a countable discrete MF group.\begin{itemize}
\item[(i)] If  $\G$ is locally $\{U_n(\C): n \in \mathbb{N}\}$-stable, then $\G$ is MAP and hence quasidiagonal.
 \item[(ii)] If $\G$ is finitely generated  and $\G$ is locally $\{GL_n(\C): n \in \mathbb{N}\}$-stable, then $\G$ is residually finite and hence quasidiagonal.
\end{itemize}
\end{lemma}
\begin{proof} (i) Since $\G$ is MF, it embeds  in $U(\prod_n M_{n}/\bigoplus_n M_{n}).$ By local $\{U_n(\C): n \in \mathbb{N}\}$-stability we obtained an embedding of $\G$ in $\prod_n U(n)$.
 (ii) By local $\{GL_n(\C): n \in \mathbb{N}\}$-stability we obtained an embedding of $\G$ in $\prod_n GL_n(\C)$. If $\G$ is finitely generated, it follows by Malcev's theorem \cite[6.4.13]{Brown-Ozawa} that  $\G$ is residually finite and hence quasidiagonal.
\end{proof}

\begin{notation}\label{not:gl-flat}
For a compact space $Y,$ let $RK^0(Y;B)_{\text{flat}}$ be the subgroup of $RK^0(Y;B)$  generated by locally trivial bundles
 with typical fiber  projective $B$-modules $fB^k$,  $f^2=f=f^*\in M_k(B),$ constructed from
 finite open covers $\mathcal{U}=(U_i)_{i\in I}$ of $Y$ and  1-Cech cocyles  $v_{ij}:U_i\cap U_j \to GL(fM_k(B)f)$  with the property that each function $v_{ij}$ is constant.
 Recall that we identify $RK^0(Y;B)$ with the operator K-theory group $K_0(C(Y)\otimes B)$.
 For  a compact subspace $Y$  of $B\G$, we denote by $\nu_Y$ the composition of the dual assembly map $\nu:KK(C^*(\G),B)\to   RK^0(B\G;B)$  with the restriction map $RK^0(B\G;B)\to RK^0(Y;B)$.
For  $B=C(\T)\otimes \QQ$, we will use the natural isomorphisms
\begin{equation}\label{eq:long}
	RK^0(Y;B)\cong K_0( C(Y)\otimes C(\T)\otimes \QQ)\cong RK^0(Y\times \T;\Q).
\end{equation}
\end{notation}
\begin{proposition}\label{prop:prop} Let $\mathcal{B}$ be a class of unital separable $C^*$-algebras such that
 if $B \in \mathcal{B}$, and $f^2=f=f^*\in M_k(B)$  for $k\in \N$, then $fM_k(B)f \in  \mathcal{B}$.
Let $\G$ be a discrete countable group and let $Y\subset B\G$ be a finite connected CW-complex. If $\G$ is locally $GL(\mathcal{B})$-stable, then for any $B\in \mathcal{B}$,
\begin{equation}\label{eqn:flatt}
	\nu_Y(KK(C^*(\G),B)_{\text{qd}})\subset RK^0(Y;B)_{\text{flat}}
\end{equation}
More generally, if $(B_k)$ in an increasing sequence of  C*-algebras  in $\mathcal{B}$ sharing the same unit and
 $B$  is the $C^*$-completion of ${\bigcup_k B_k},$ then
\begin{equation}\label{eqn:flattt}
	\nu_Y(KK(C^*(\G),B)_{\text{qd}})\subset \varinjlim\nolimits_{k} RK^0(Y;B_k)_{\text{flat}} \subset RK^0(Y;B)_{\text{flat}}.
\end{equation}
	\end{proposition}
	\begin{proof}   We prove \eqref{eqn:flatt} first. Let
	$\jmath:\G\hookrightarrow U(C^*(\G))$ be the natural inclusion.
	The orbit space of the left action of  $\G$ on $\widetilde{E\G}\times C^*(\G)$, defined by $s\cdot (\tilde{p},v)=(s\cdot \tilde{p},\jmath(s)v)$,
  is the total space of  a (flat) Hilbert $C^*(\G)$-module bundle $E\G\times_\jmath C^*(\G) \to  B\G$ denoted  by $L_\jmath$.

  The restriction of $L_\jmath$ to $Y$, denoted $L_Y,$ yields a self-adjoint projection $p=p_Y$
in matrices over the ring $ C(Y)\otimes \C[\G]$  constructed as follows. Let  $(U_i)_{i\in I}$ be finite covering of $Y$ by open sets
such that $L_\jmath$  is trivial on each $U_i$ and $U_i\cap U_j$ is connected. Using  trivializations
of $L_\jmath$  on $U_i$ one obtains group elements $s_{ij}\in G$
 which define  a 1-cocycle that is constant on each nonempty set $U_i\cap U_j$  and which represents $L_Y$.  Thus $s_{ij}^{-1}=s_{ji}$ and $s_{ij} \cdot s_{jk}=s_{ik}$ whenever $U_i\cap U_j \cap U_k \neq \emptyset.$
 Let $(\chi_i)_{i\in I}$
be positive continuous functions with $\chi_i$ supported in $U_i$ and such that
$\sum_{i\in I} \chi^2_i=1$. Set $m=|I|$ and let $(e_{ij})$ be the canonical matrix unit of $M_m(\CCC)$. Then $L_Y$ is represented by the selfadjoint projection
\begin{equation}\label{eqn-mish}
 p=\sum_{i,j\in I} e_{ij}\otimes \chi_i\chi_j\otimes s_{ij}\in M_{m}(\CCC)\otimes C(Y)\otimes C^*(\G).
\end{equation}
It was shown by Kasparov \cite[Lemma~6.2]{Kas:inv}, \cite{Kasparov:conspectus},
that the map \[\nu_Y:KK(C^*(\G),B) \stackrel{\nu}\longrightarrow K^0(B\G;B) \to K^0(Y;B) \cong K_0(C(Y)\otimes B)\] is given by $\nu_Y(x)=[p]\otimes_{C^*(\G)} x$.

The restriction of $\nu_Y$ to quasidiagonal KK-classes can be described as follows, see  \cite{AA}.
Each element $x\in KK(C^*(\G),B)_{qd}$
 is represented
  by  a pair of nonzero $*$-representations $\Phi^{(r)}:C^*(\G)\to M(K(H)\otimes B)$,  $r=0,1$, such that $\Phi^{(0)}(a)-\Phi^{(1)}(a)\in K(H)\otimes B$,
 $a\in C^*(\G)$, and with the property that
 there is an increasing approximate unit $(p_n)_n$ of $K(H)$ consisting of projections such that $(p_n\otimes 1_B)_n$ commutes asymptotically with both $\varphi^{(0)}(a)$ and $\varphi^{(1)}(a)$, for all $a\in C^*(\G)$.
  It follows that the compressions $\varphi^{(r)}_n(\cdot)=(p_n\otimes 1_B) \Phi^{(r)}(\cdot)(p_n\otimes 1_B)$
 are completely positive asymptotic homomorphisms $\varphi^{(r)}_n:C^*(\G)\to K(H)\otimes B$, $r=0,1$. Let $1$ denote the unit of $C^*(\G)$.
It is routine to  further perturb these maps to completely positive asymptotic homomorphisms such that
$f^{(r)}_n:=\varphi^{(r)}_n(1)$ are selfadjoint projections in matrices over $B$. Hence we can view these maps as unital completely positive maps $\varphi^{(r)}_n:C^*(\G) \to  D^{(r)}_n$, where each $D^{(r)}_n=f^{(r)}_n (K(H)\otimes B)f^{(r)}_n$ is Morita equivalent to $B$.
 As argued in  \cite[Prop.2.5]{AA}, if $\id{}=\id{M_m(C(Y))},$
 \begin{equation}\label{eqn:important}
  \nu_Y(x)=[p]\otimes_{C^*(\G)} x = (\id{}\otimes\varphi^{(0)}_n)_\sharp(p)-( \id{}\otimes\varphi^{(1)}_n)_\sharp(p),
 \end{equation}
 for all sufficiently large $n$.
Here $(\id{}\otimes\varphi^{(r)}_n)_\sharp(p)\in K_0(C(Y)\otimes B)$ is the  class of the  projection
$\chi_{(1/2,1]}(x_n^{(r)})$ obtained by continuous functional calculus from the approximate projection $x_n^{(r)}=\sum_{i,j\in I} e_{ij}\otimes \chi_i\chi_j\otimes \varphi^{(r)}_n(s_{ij})$.
Since $\varphi^{(r)}_n(s)$ are almost unitary elements and since $\G$ is countable, there exist
sequences of unital maps $\sigma^{(r)}_n: \G \to U(D^{(r)}_n)$ such that $\lim_n\|\varphi^{(r)}_n(s)-\sigma^{(r)}_n(s)\|=0$ for all $s\in \G$.
The sequences $(\sigma^{(r)}_n),$ $r=0,1$ are asymptotically multiplicative in the sense of \eqref{a1}.
If $\G$ satisfies Definition~\ref{def:B-stable} (b), then
 there exist sequences of  group homomorphisms  $\{\pi^{(r)}_n:\G \to GL(D_n^{(r)})\}$, $r=0,1$, such that
\begin{equation}\label{eqn:importanti}\lim_n\|\sigma^{(r)}_n(s)-\pi^{(r)}_n(s)\|=0 \end{equation}
 for all $s\in \G$.

Note that the projections $e^{(r)}_n$,  $r=0,1$, defined by
 \[e^{(r)}_n=(\id{}\otimes\pi^{(r)}_n)(p)=\sum_{i,j\in I} e_{ij}\otimes \chi_i\chi_j\otimes \pi^{(r)}_n(s_{ij})\in M_{m}(\CCC)\otimes C(Y)\otimes D_n^{(r)},\]
 correspond to flat  bundles since they are realized via the constant cocycles $\pi^{(r)}_n(s_{ij})$ and hence $[e^{(0)}_n]-[e^{(1)}_n]\in RK^0(Y,B)_{\text{flat}}$.
 From \eqref{eqn-mish}, \eqref{eqn:important} and \eqref{eqn:importanti} we deduce that
 $\nu_Y(x)= [e^{(0)}_n]-[e^{(1)}_n]$ for all sufficiently large $n$, since $\|\sigma^{(r)}_n(s_{ij})-\pi^{(r)}_n(s_{ij})\|\to 0$. We conclude that $\nu_Y(x)\in RK^0(Y,B)_{\text{flat}}$, as desired.

The proof for \eqref{eqn:flattt} is similar.
One makes a small adjustment to the previous argument.
 By assumption, ${\bigcup_k B_k}$ is dense in $B$.
Enumerate $\G=\{g_n\}_n$ and  modify  the maps  $(\sigma^{(r)}_n)$, $r=0,1$ as follows.
Choose selfadjoint projections $h^{(r)}_n$ in $K(H)\otimes B_{k{(n,r)}}$  for sufficiently large  $k{(n,r)}\in \N$
such that $h^{(r)}_n$ approximate $f^{(r)}_n$ and  $\mathrm{dist}(\varphi_n^{(r)}(s), U(h^{(r)}_n (K(H)\otimes B_{k{(n,r)}})h^{(r)}_n))<1/n $ for all $s\in \{g_1,...,g_n\}$.
Letting $D_n^{(r)}=h^{(r)}_n (K(H)\otimes B_{k{(n,r)}})h^{(r)}_n$, we can now
construct   $\sigma^{(r)}_n:\G \to U (D_n^{(r)})$ with the property that  $\lim_n\|\varphi^{(r)}_n(s)-\sigma^{(r)}_n(s)\|=0$ for all $s\in \G$. The rest of the proof is as in (i).
Note that $[e^{(0)}_n]-[e^{(1)}_n]\in RK^0(Y,B_{k(n,r)})_{\text{flat}}$.
 \end{proof}
 Let $X$ be a finite connected CW complex. Let $\G$ be a discrete countable group and let $\rho:\G \to GL_r(C(X))$ be a group homomorphism. Consider the flat bundle $L_\rho$ defined as $E\G\times_\rho C(X)^r \to B\G$ whose typical fiber is $C(X)^r$.
 Let $Y\subset B\G$ be a finite CW complex. We can view the restriction of $L_{\rho}$  to $Y$, denoted $L_{\rho}|_Y,$
 as a rank $r$ complex vector bundle over $Y\times X$.
 \begin{theorem}[Baird-Ramras, \cite{Ramras}]\label{Baird-Ramras}
 Suppose that $H^k(X;\Q)=0$ for $k>d$
  Then for all $m>0$, the Chern classes  $c_{d+m} (L_{\rho}|_Y)\in H^{2d+2m}(Y\times X;\Z)$ map to zero in $H^{2d+2m}(Y\times X;\Q)$.
  \end{theorem}
We shall apply Theorem~\ref{Baird-Ramras} for $X=\{\text{point}\}$ with $d=0$ and for $X=\T$ with $d=1$.
\begin{lemma}\label{lem:prod}
Let $X$ be a  topological space with $H^{2}(X,\Q)=0$ and such that  $H^{even}(X\times \TT,\Q)$ is generated as a ring by
$H^{0}(X\times \TT,\Q)\oplus H^{2}(X\times \TT,\Q)$. Then $H^{k}(X,\Q)= 0$ for all $k> 1$.
\end{lemma}
\begin{proof}
Let $\pi_1:X\times \T \to X$ and $\pi_2:X\times \T \to \TT$ be the canonical projections.
By\cite[Thm.61.6]{Munkres},  the cross product $(x,y)\mapsto x\times y=\pi_1^*(x)\cup \pi_2^*(y)$  induces an isomorphism of algebras
\begin{equation}\label{eq:theta}
	\theta: H^{*}(X,\Q)\otimes_{_\Q} H^{*}(\TT,\Q) \to H^{*}(X\times \TT,\Q).\end{equation}
Since $H^{2}(X,\Q)=0$, we must have
\(H^{1}(X,\Q)\otimes_{_\Q} H^{1}(\TT,\Q)\cong H^{2}(X\times \TT,\Q).\)

 By \cite[Thm.61.5]{Munkres}, in the cohomology ring $H^*(X\times \TT,\Q)$, we have \[(\alpha\times\beta)\cup (\alpha'\times\beta')=(-1)^{(deg\beta)(deg\alpha')}(\alpha\cup \alpha')\times (\beta \cup \beta')\]
   for $\alpha,\alpha'\in H^*(X,\Q)$ and $\beta,\beta'\in H^*(\TT,\Q)$.
    It follows that in our situation $\gamma \cup \gamma' =0$ for all
   $\gamma,\gamma'\in H^{2}(X\times \TT,\Q) $ as  $\beta\cup \beta' =0$ for
   $\beta,\beta'\in H^{1}( \TT,\Q) $. Since $H^{even}(X\times \TT,\Q)$ is generated as a ring by $H^{0}(X\times \TT,\Q)\oplus H^{2}(X\times \TT,\Q)$ we deduce that $H^{2k}(X\times \TT,\Q)=0$ for all $k\geq 2$. Since
   \[H^{2k}(X,\Q) \oplus  H^{2k-1}(X,\Q) \cong  H^{2k}(X\times \TT,\Q)\]
   by \eqref{eq:theta}, it follows that $H^{j}(X,\Q)=\{0\}$ for all $j\geq 2$.
\end{proof}

 The following is the main result of the second part of our paper.
 Its first part (i) strengthens (at least formally) the main result of \cite{CCC}, since in principle, local $\{GL_n(\C): n \in \mathbb{N}\}$-stability is weaker  than local $\{U_n(\C): n \in \mathbb{N}\}$-stability.
\begin{theorem}\label{thm:2nd}
  Let $\G$ be a  countable discrete  group that  admits a $\gamma$-element. Suppose that  $\G$ is either MF and finitely generated  or  that $\G$ is quasidiagonal.
  \begin{itemize}
\item[(i)]If $\G$ is locally  $\{GL_n(\C): n \in \mathbb{N}\}$-stable, then  $H^{2k}(\G,\Q)=0$ for all $k>0$.
\item[(ii)] If $\G$ is  locally $\{GL_n(C(\T)): n \in \mathbb{N}\}$-stable, then  $H^{k}(\G,\Q)=0$ for all $k>1$.
\end{itemize}
\end{theorem}
\begin{remark}\label{rem13}
 Theorem~\ref{thm:2nd} applies to finitely generated linear groups and to residually finite hyperbolic groups, among others.  Indeed, linear groups are MF by \cite{CCC} and exact by \cite{GuenHW} and hence they admit a $\gamma$-element by \cite{Tu:gamma}.
In particular, since it was shown in  \cite{BLSW} that if $\G$ is a cocompact lattice in a real semisimple Lie group $G$ which is not locally isomorphic to either $SO(n,1)$ for $n$ odd or $SL_3(\R)$, then $b_{2i}(\G)>0$ for some $i>0$, and since $\G$ is finitely generated by cocompactness,
 we deduce that $\G$ is not $\{GL_n(\C): n \in \mathbb{N}\}$-stable.  Hyperbolic groups are exact by \cite{Brown-Ozawa}.
It is not known whether there exists a hyperbolic group which is not residually finite.

\end{remark}
\begin{proof} (for Thm.~\ref{thm:2nd})
By Lemma~\ref{lem:easy} and Remark~\ref{remm} (iv),  MF finitely generated groups that satisfy the assumptions from (i) or from (ii) are quasidiagonal. Thus we may assume that $\G$ is quasidiagonal for the remainder of the proof.

Let $B=\QQ$ or $B=C(\T)\otimes \QQ$.
By Theorem~\ref{cor:qd-embed-kk} we have a surjective map
 \[\nu:KK(C^*(\G),B)_{qd}\to RK^0(B\G;B) .\]
 If $B\G$ is written as the union of an increasing sequence
$(Y_i)_{i}$ of finite CW complexes, then as explained in the proof of Lemma 3.4 from \cite{Kasparov-Skandalis-kk}, there is a Milnor $\varprojlim^1$ exact sequence which gives
\begin{equation}\label{eq:milnor}
RK^0(B\G;B)\cong \varprojlim RK^0(Y_i;B)
\end{equation} since $K_0(B)$ is divisible. We denote by $\nu_i$ the composition of the map $\nu$ defined above with the restriction map $RK^0(B\G;B)\to RK^0(Y_i;B)$.

(i)  For $B=\QQ$, using the naturality of the Chern character, we have a commutative diagram
 \[\xymatrix{
{KK(C^*(\G),\QQ)_{qd}}\ar[dr]_-{(\nu_i)} \ar[r]^-{\nu}&{RK^0(B\G; \Q)}\ar[d]_{}\ar[r]^{ch}& {{H^{even}(B\G; \Q)}}\ar[d]^{}
\\
{}& {\varprojlim RK^0(Y_i; \Q)}\ar[r]^{ch}  & {\varprojlim{H^{even}(Y_i; \Q)}}
}
\] with bijective vertical maps. Recall that we identify $RK^0(Y;\Q)$ with $RK^0(Y;\QQ)$, \cite{Kasparov-conspectus}. Write $\QQ$ as inductive limit of matrix algebras $B_k\cong M_{k!}(\C)$. By Proposition~\ref{prop:prop}, the image of each $\nu_i$ is contained in $\varinjlim\nolimits_{k} RK^0(Y_i;B_k)_{\text{flat}}$ and hence by Theorem~\ref{Baird-Ramras} applied for $X=\{\text{point}\},$ the image of the map $ch\circ \nu_i: RK^0(Y_i; \Q) \to H^{even}(Y_i; \Q)$ is contained in $H^{0}(Y_i; \Q)$. Since $\nu$ is surjective, it follows  that $ch(RK^0(B\G; \Q))=H^{even}(B\G; \Q)$ and hence
$H^{2k}(\G,\Q)=0$ for all $k\geq 1$.
\vskip 4pt
(ii) Now let $B=C(\T)\otimes \QQ$ and write $B$ as the inductive limit of $B_k=C(\T)\otimes M_{k!}(\C)$.
  Just as above, we denote by $\nu_i$ the composition of the map $\nu$  with the restriction map $RK^0(B\G;B)\to RK^0(Y_i;B)\cong RK^0(Y_i\times \T; \Q)$. By \eqref{eq:long} and \eqref{eq:milnor}, we have
  \begin{align*} RK^0(B\G;B)=RK^0(B\G;C(\T)\otimes\mathcal{Q})\cong \varprojlim RK^0(Y_i;C(\T)\otimes \QQ) &\cong& \\ \varprojlim RK^0(Y_i\times \T; \QQ)\cong \varprojlim RK^0(Y_i\times \T; \Q)\cong RK^0(B\G\times \T; \Q).
  \end{align*}
  Abusing the notation, we denote the map $KK(C^*(\G),B)_{qd} \to RK^0(B\G\times \T; \Q)$  obtained by composing $\nu$ with the isomorphism  $RK^0(B\G;B) \cong RK^0(B\G\times \T; \Q)$ from above, again by $\nu$.

  We are going to argue that
each map $\nu_i:KK(C^*(\G),B)_{qd} \to  RK^0(Y_i\times \T; \Q)$
  has the property that the rational Chern classes satisfy $c_m(\nu_i(y))=0$ for all $y\in KK(C^*(\G),B)_{qd}$ and $m\geq 2$. Indeed, by
  Proposition~\ref{prop:prop}, the image of each $\nu_i$ is contained in $\varinjlim\nolimits_{k} RK^0(Y_i;B_k)_{\text{flat}}$ and hence by Theorem~\ref{Baird-Ramras} applied for $X=\T,$ we deduce the desired property.
It follows that,  for all $y\in KK(C^*(\G),B)_{qd},$ the Chern character of $\nu_i(y)$ can be computed as
\[
ch(\nu_i(y))=exp(c_1(\nu_i(y))).
\]
  Using the commutative diagram with bijective vertical arrows
  \[\xymatrix{
{KK(C^*(\G),C(\T)\otimes \QQ)_{qd}}\ar[dr]_-{(\nu_i)} \ar[r]^-{\nu}&{RK^0(B\G\times \T; \Q)}\ar[d]_{}\ar[r]^{ch}& {{H^{even}(B\G\times \T; \Q)}}\ar[d]^{}
\\
{}& {\varprojlim RK^0(Y_i\times \T; \Q)}\ar[r]^{ch}  & {\varprojlim{H^{even}(Y_i\times \T; \Q)}}
}
\]
we deduce that $ch(\nu(y))=exp(c_1(\nu(y)))$ for all $y\in KK(C^*(\G),B)_{qd}$ and hence
$ch(z)=exp(c_1(z))$ for all $z\in RK^0(B\G\times \T; \Q)$, by surjectivity of $\nu$. Since the Chern character is a rational isomorphism
 and the image of the first Chern class has degree two,  it follows that
$H^{even}(B\G\times \TT,\Q)$ is generated as a ring by
$H^{0}(B\G\times \TT,\Q) \oplus H^{2}(B\G\times \TT,\Q)$. The conclusion of the theorem follows now from Lemma~\ref{lem:prod}
since $H^*(B\G,\Q)\cong H^*(\G,\Q)$ and we already know that  $H^2(B\G,\Q)=0$ by part (i) above and Remark~\ref{remm} (iv).
\end{proof}
 
 The first part of the theorem below was already proved in our earlier paper \cite{CCC}. We include the statement for the sake of completeness.
\begin{theorem}\label{cor2nd}
  Let $\G$ be an MF countable discrete  group that  admits a $\gamma$-element.
  \begin{itemize}
\item[(i)]If $\G$ is  locally $\{U_n(\C): n \in \mathbb{N}\}$-stable, then  $H^{2k}(\G,\Q)=0$ for all $k>0$.
\item[(ii)] If $\G$ is  locally $\{U_n(C(\T)): n \in \mathbb{N}\}$-stable, then  $H^{k}(\G,\Q)=0$ for all $k>1$.
\end{itemize}
\end{theorem}
\begin{proof} By Lemma~\ref{lem:easy}(i), $\G$ is quasidiagonal. Since local $U(\mathcal{B})$-stability implies local $GL(\mathcal{B})$-stability, Theorem~\ref{cor2nd} follows from Theorem~\ref{thm:2nd}.
\end{proof}
%\bibliographystyle{abbrv}
%\bibliography{operator22-basic}
\end{document}